\pdfoutput=1
\documentclass[11pt,a4paper]{amsart}
\usepackage[margin=1in]{geometry}
\usepackage[T1]{fontenc}
\usepackage{textcase}
\usepackage[dvipsnames]{xcolor}
\usepackage{microtype}
\usepackage{fnpct}

\usepackage{amsmath}
\usepackage{amssymb}
\usepackage{eucal}
\usepackage{mathrsfs}
\usepackage{tikz-cd}
\renewcommand{\injlim}{\varinjlim}
\renewcommand{\projlim}{\varprojlim}

\usepackage{enumitem}
\usepackage{booktabs}
\usepackage[pdfusetitle,colorlinks]{hyperref}
\hypersetup{bookmarksdepth=2,pdfencoding=unicode,allcolors=MidnightBlue}
\usepackage[capitalise,noabbrev,nosort]{cleveref}

\crefname{equation}{}{}
\crefformat{enumi}{(#2#1#3)}
\crefname{enumi}{}{}
\newlist{conenum}{enumerate}{1}
\setlist[conenum,1]{label=(\roman*),ref=\roman*}
\creflabelformat{conenumi}{(#2#1#3)}
\crefname{conenumi}{}{}

\numberwithin{equation}{section}
\theoremstyle{plain}
\newtheorem{Theorem}{Theorem}

\crefname{Theorem}{Theorem}{Theorems}
\newtheorem{theorem}[equation]{Theorem}
\newtheorem{proposition}[equation]{Proposition}
\newtheorem{lemma}[equation]{Lemma}
\newtheorem{corollary}[equation]{Corollary}

\theoremstyle{definition}
\newtheorem{definition}[equation]{Definition}

\newtheorem{example}[equation]{Example}

\theoremstyle{remark}
\newtheorem{remark}[equation]{Remark}

\let\oldSS\SS\let\SS\relax

\newcommand{\NN}{\mathbf{N}}
\newcommand{\ZZ}{\mathbf{Z}}
\newcommand{\QQ}{\mathbf{Q}}

\newcommand{\CC}{\mathbf{C}}
\newcommand{\SS}{\mathbf{S}}

\newcommand{\II}{\mathbb{I}}

\newcommand{\E}{\mathrm{E}}

\newcommand{\cg}{\textnormal{cg}}
\newcommand{\con}{\textnormal{con}}

\newcommand{\fin}{\textnormal{fin}}

\newcommand{\rig}{\textnormal{rig}}

\newcommand{\st}{\textnormal{st}}
\newcommand{\vSc}{\textnormal{vSc}}

\newcommand{\vnuc}{\textnormal{vnuc}}

\newcommand{\CAlg}{\operatorname{CAlg}}
\newcommand{\D}{\operatorname{D}}
\newcommand{\Frac}{\operatorname{Frac}}
\newcommand{\VD}{\operatorname{VD}}
\newcommand{\Fun}{\operatorname{Fun}}
\newcommand{\Idem}{\operatorname{Idem}}
\newcommand{\cIdem}{\operatorname{cIdem}}

\newcommand{\Ind}{\operatorname{Ind}}

\newcommand{\Map}{\operatorname{Map}}
\newcommand{\Mod}{\operatorname{Mod}}

\newcommand{\Aut}{\operatorname{Aut}}
\newcommand{\Pro}{\operatorname{Pro}}

\newcommand{\R}{\operatorname{R}}
\newcommand{\Shv}{\operatorname{Shv}}
\newcommand{\Seq}{\operatorname{Seq}}
\newcommand{\cShv}{\operatorname{cShv}}
\newcommand{\Sm}{\operatorname{Sm}}
\newcommand{\Spec}{\operatorname{Spec}}

\newcommand{\cofib}{\operatorname{cofib}}
\newcommand{\fib}{\operatorname{fib}}
\newcommand{\id}{\operatorname{id}}
\newcommand{\map}{\operatorname{map}}
\newcommand{\op}{\operatorname{op}}
\newcommand{\co}{\operatorname{co}}
\newcommand{\sd}{\operatorname{sd}}

\newcommand{\X}{\mathord{-}}
\newcommand{\unit}{\mathbf{1}}

\newcommand{\liq}{\operatorname{liq}}

\newcommand{\cat}[1]{\mathcal{#1}}
\newcommand{\Cat}[1]{\mathsf{#1}}
\newcommand{\shf}[1]{\mathscr{#1}}

\newcommand{\Cls}[1]{\mathscr{#1}}

\newcommand{\llbracket}{[\![}
\newcommand{\rrbracket}{]\!]}

\usepackage{mathtools}

\DeclarePairedDelimiterX\Set[1]{\{}{\}}{#1}

\title{Very Schwartz coidempotents and continuous spectrum}
\author{Ko Aoki}
\address{Max Planck Institute for Mathematics,
  Vivatsgasse 7, 53111 Bonn, Germany
}
\email{aoki@mpim-bonn.mpg.de}
\date{\today}

\begin{document}

\begin{abstract}
  We introduce the continuous version
  of the (unstable) smashing spectrum functor.
  In the stable case,
  it assigns to each dualizably symmetric monoidal
  stable presentable \(\infty\)-category
  a stably compact space
  whose open subsets
  correspond to very Schwartz idempotents—a certain class
  of idempotents we define.
  As an application,
  we prove Tannaka duality for
  spectral sheaves on stably compact spaces,
  including the case of compact Hausdorff spaces.
\end{abstract}

\maketitle
\setcounter{tocdepth}{1}
\tableofcontents

\section{Introduction}\label{s:intro}

\subsection{Tannaka duality for compact Hausdorff spaces}\label{ss:tan_ch}

Classical Tannaka duality
recovers a group~\(G\)
from its category of representations.
Lurie~\cite{Lurie05}
considered the problem
of reconstructing a stack~\(X\)
from its symmetric monoidal \(\infty\)-category of quasicoherent sheaves~\(\D(X)\),
which is related to the original Tannaka duality
when \(X\) is~\(BG\);
see~\cite[Chapter~9]{LurieSAG} for an extensive discussion.
Similarly,
one can ask whether a topological space
can be recovered from its category of sheaves.

In~\cite{ttg-shv},
we observed
a result of this type
for locales:
For a field~\(k\),
the functor
\begin{equation}
  \label{e:bf9ux}
  \Shv(\X;\D(k))\colon
  \Cat{Loc}^{\op}
  \to
  \CAlg_{\D(k)}(\Cat{Pr}_{\st}),
\end{equation}
where \(\Cat{Loc}\) denotes the category of locales,
is fully faithful;
in particular,
any \(k\)-linear symmetric monoidal functor
\(\Shv(X;\D(k))\to\Shv(Y;\D(k))\)
for locales~\(X\) and~\(Y\)
comes from a morphism \(Y\to X\).
The key point of our proof
was to consider the right adjoint of~\cref{e:bf9ux}:
In~\cite{ttg-sm},
we identified
the classical \emph{smashing spectrum} functor~\(\Sm\)
with the right adjoint of
\begin{equation*}
  \Shv(\X;\Cat{Sp})\colon
  \Cat{Loc}^{\op}
  \to
  \CAlg(\Cat{Pr}_{\st}).
\end{equation*}
Therefore,
to deduce that \cref{e:bf9ux} is fully faithful,
it suffices to show that
the map of locales
\(\Sm(\Shv(X;\D(k)))\to X\)
is an equivalence for a locale~\(X\)—reducing
the problem to understanding idempotent algebras
in \(\Shv(X;\D(k))\).

In this paper,
we prove that
\begin{equation}
  \label{e:6h4hz}
  \Shv(\X;\Cat{Sp})\colon
  \Cat{CH}^{\op}
  \to
  \CAlg(\Cat{Pr}_{\st}),
\end{equation}
where \(\Cat{CH}\) denotes the category of compact Hausdorff spaces,
is fully faithful;
in particular,
any colimit-preserving symmetric monoidal functor
\(\Shv(X;\Cat{Sp})\to\Shv(Y;\Cat{Sp})\)
for compact Hausdorff spaces~\(X\) and~\(Y\)
comes from a continuous map \(Y\to X\).
The plain smashing spectrum functor~\(\Sm\)
does not help in this case,
since \(\Sm(\Cat{Sp})\)
is not a singleton (and our understanding is limited,
but see~\cite{BHIS} for recent progress).
Rather, we prove this
by considering a variant of the smashing spectrum functor.

First,
we note that by, e.g.,~\cite[Theorem~6.5]{verdier-asc},
the functor
\cref{e:6h4hz} lands in
the full subcategory
\(\CAlg(\Cat{Pr}_{\st})_{\rig}\)
of rigid stable presentably symmetric monoidal \(\infty\)-categories.

\begin{Theorem}\label{main_rig}
  We have an adjunction
  \begin{equation*}
    \begin{tikzcd}[column sep=huge]
      \Cat{CH}^{\op}\ar[r,shift left,"\Shv(\X;\Cat{Sp})"]&
      \CAlg(\Cat{Pr}_{\st})_{\rig}
      \subset
      \CAlg(\Cat{Pr}_{\st})
      \rlap,\ar[l,shift left,"\Sm^{\rig}"]
    \end{tikzcd}
  \end{equation*}
  where the right adjoint~\(\Sm^{\rig}\)
  is called the \emph{rigid spectrum} functor.
  For an object \(\cat{C}\in\CAlg(\Cat{Pr}_{\st})_{\rig}\),
  a closed subset of
  \(\Sm^{\rig}(\cat{C})\)
  corresponds to a \emph{very nuclear idempotent},
  which we introduce in \cref{ss:nr}.
\end{Theorem}

The existence of the right adjoint
in \cref{main_rig} is straightforward;
one can define
\(\Sm^{\rig}(\cat{C})\)
as the Stone–Čech compactification
of \(\Sm(\cat{C})\).
What is important to us
is to describe the open subsets of~\(\Sm^{\rig}(\cat{C})\)
in terms of special idempotents.
Using our presentation,
we prove the following,
which implies the desired Tannaka duality statement:

\begin{Theorem}\label{tan_ch}
  The rigid spectrum
  of \(\Shv(X;\Cat{Sp})\)
  for a compact Hausdorff space~\(X\)
  is canonically identified with~\(X\);
  i.e., \cref{e:6h4hz} is fully faithful.
\end{Theorem}

\begin{remark}\label{xgldwx}
  Since \(\Cat{CH}\)
  is generated under limits by \([0,1]\),
  we can restate \cref{main_rig,tan_ch} as follows:
  Inside \(\CAlg(\Cat{Pr}_{\st})_{\rig}\),
  we consider the full subcategory generated
  under colimits
  by \(\Shv([0,1];\Cat{Sp})\).
  Then this full subcategory is equivalent
  to \(\Cat{CH}^{\op}\) via \(\Shv(\X;\Cat{Sp})\)
  and is coreflective.
  The coreflector maps \(\cat{C}\)
  to \(\Shv(\Sm^{\rig}(\cat{C});\Cat{Sp})\).
\end{remark}

\subsection{Stably compact spaces and continuous spectrum}\label{ss:intro_sc}

Our main object of study
in this paper is the notion
of \emph{continuous} spectrum
rather than that of rigid spectrum.
To explain this,
we first need to recall the notion of stably compact spaces,
which we review in \cref{s:sc}.
It is a simultaneous generalization of spectral spaces
and compact Hausdorff spaces;
see \cref{f:006w8}.

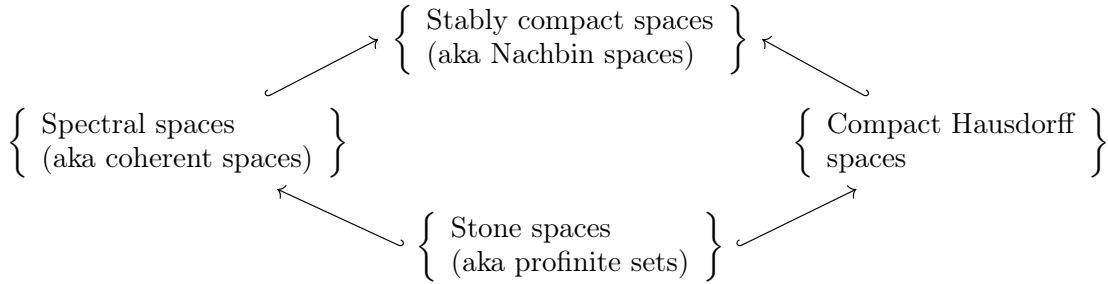
\begin{figure}[htbp]
  \begin{equation*}
    \begin{tikzcd}[column sep=tiny,row sep=tiny]
      {}&
      \Set*{\text{\begin{tabular}{l}
            Stably compact spaces\\
            (aka Nachbin spaces)
      \end{tabular}}}&
      {}\\
      \Set*{\text{\begin{tabular}{l}
            Spectral spaces\\
            (aka coherent spaces)
      \end{tabular}}}
      \ar[ur,hook,end anchor=west]&
      {}&
      \Set*{\text{\begin{tabular}{l}
            Compact Hausdorff\\
            spaces
      \end{tabular}}}
      \ar[ul,hook',end anchor=east]\\
      {}&
      \Set*{\text{\begin{tabular}{l}
            Stone spaces\\
            (aka profinite sets)
      \end{tabular}}}
      \ar[ul,hook',start anchor=west]
      \ar[ur,hook,start anchor=east]&
      {}
    \end{tikzcd}
  \end{equation*}
  \caption{The class of stably compact spaces contains
    both
    the classes of spectral spaces and compact Hausdorff spaces.
    The intersection of these two is the class of Stone spaces.
  }\label{f:006w8}
\end{figure}

We also have generalizations
of familiar notions in this class;
see \cref{t:4zbff}:

\begin{table}[htbp]
  \centering
  \caption{Notions on stably compact spaces.}\label{t:4zbff}
  \begin{tabular}{lll}
    \toprule
    Spectral spaces&
    Stably compact spaces&
    Compact Hausdorff spaces\\
    \midrule
    quasicompact map&
    perfect map&
    continuous map\\
    Hochster dual&
    de~Groot dual&
    identity\\
    constructible topology&
    patch topology&
    identity\\
    conj.\ of
    \(\Pro\)-Zariski top.&
    Flachsmeyer construction&
    N/A\\
    \bottomrule
  \end{tabular}
\end{table}

For sheaf theory,
we prove
in \cref{ss:sheaf}
that
\(\Shv(\X;\Cat{Sp})\colon\Cat{Loc}^{\op}\to\CAlg(\Cat{Pr}_{\st})\)
restricts to
\(\Cat{SC}^{\op}\to\CAlg(\Cat{Dbl})\),
where \(\Cat{SC}\) and \(\Cat{Dbl}\)
denote the \(\infty\)-categories
of stably compact spaces 
and dualizable stable presentable \(\infty\)-categories,
respectively.
With this,
we state
our main construction as follows:

\begin{Theorem}\label{main_dbl}
  We have an adjunction
  \begin{equation*}
    \begin{tikzcd}[column sep=huge]
      \Cat{SC}^{\op}\ar[r,shift left,"\Shv(\X;\Cat{Sp})"]&
      \CAlg(\Cat{Dbl})
      \rlap.\ar[l,shift left,"\Sm^{\con}"]
    \end{tikzcd}
  \end{equation*}
  where the right adjoint~\(\Sm^{\con}\)
  is called the \emph{continuous spectrum} functor.
  For an object \(\cat{C}\in\CAlg(\Cat{Dbl})\),
  a closed subset of~\(\Sm^{\con}(\cat{C})\)
  corresponds to a \emph{very Schwartz idempotent},
  which we introduce in \cref{ss:vsc}.
\end{Theorem}

Note that this is subtler than \cref{main_rig},
since neither side is a full subcategory
of \(\Cat{Loc}^{\op}\) or \(\CAlg(\Cat{Pr}_{\st})\),
respectively.

\begin{remark}\label{x5cdq5}
  As we did in~\cite{ttg-sm},
  we prove the unstable version of \cref{main_dbl}
  by considering very Schwartz \emph{coidempotents}.
\end{remark}

We also prove
the following strengthening of \cref{tan_ch}
in \cref{ss:sm_sc}:

\begin{Theorem}\label{tan_sc}
  The continuous spectrum
  of \(\Shv(X;\Cat{Sp})\)
  for a stably compact space~\(X\)
  is canonically identified with~\(X\),
  i.e.,
  the left adjoint in \cref{main_dbl}
  is fully faithful.
\end{Theorem}

\begin{remark}\label{x2mvyj}
  In~\cite[Appendix~F]{Efimov},
  Efimov drew an analogy
  between compact Hausdorff spaces
  and dualizable stable presentable \(\infty\)-categories.
  This paper uses a different analogy shown in \cref{t:cxic0}.
\end{remark}

\begin{table}[htbp]
  \centering
  \caption{Our analogy between locales and objects of \(\CAlg(\Cat{Pr}_{\st})\),
    where \(\Cat{Spec}\) and \(\Cat{Stone}\) denote
    the categories of spectral and Stone spaces, respectively.
  }\label{t:cxic0}
  \begin{tabular}{lll}
    \toprule
    Subcats.\ of \(\Cat{Loc}\)&
    Subcats.\ of \(\CAlg(\Cat{Pr}_{\st})\) \\
    \midrule
    \(\Cat{SC}\)&
    \(\CAlg(\Cat{Dbl})\)\\
    \(\Cat{CH}\)&
    \(\CAlg(\Cat{Pr}_{\st})_{\rig}\)\\
    \(\Cat{Spec}\)&
    \(\CAlg(\Cat{Pr}^{\cg}_{\st})\)\\
    \(\Cat{Stone}\)&
    \(\CAlg(\Cat{Pr}^{\cg}_{\st})_{\rig}\)\\
    \bottomrule
  \end{tabular}
\end{table}

\begin{remark}\label{x8h0yu}
\Cref{t:cxic0}
  suggests considering the additional spectrum functors
  \(\CAlg(\Cat{Pr}^{\cg}_{\st})\to\Cat{Spec}^{\op}\)
  and
  \(\CAlg(\Cat{Pr}^{\cg}_{\st})_{\rig}\to\Cat{Stone}^{\op}\).
  The first (or rather its unstable generalization)
  was already considered in~\cite[Remark~3.32]{ttg-sm}:
  The procedure there produces a distributive lattice,
  and we then take its spectrum
  to obtain a spectral space.
By adjunction,
  the second functor is obtained by applying~\(\pi_{0}\) to the first.
  However,
  as shown in \cref{xqp9qd},
  this is not necessary;
  the first functor always produces a Stone space
  in this case.
\end{remark}

\subsection*{Organization}\label{ss:outline}

\Cref{s:ind} introduces the notion of restricted \(\Ind\)-objects.
\Cref{s:con} reviews
the theory of compact morphisms
and continuous presentable \(\infty\)-categories.
\Cref{s:sc} studies stably compact spaces
and sheaves on them, including a digression on
the proper base change theorem and Verdier duality.
\Cref{s:sm_con} introduces the continuous spectrum functor \(\Sm^{\con}\)
and the unstable version of \cref{main_dbl}.
\Cref{s:ex} contains computations of continuous spectra,
including the proof of \cref{tan_sc}.
\Cref{s:rig_sp} turns to rigid spectra
and establishes \cref{main_rig,tan_ch}.

\subsection*{Acknowledgments}\label{ss:ack}

I thank Alexander Efimov
for explaining his work
on \(K\)-theory
of dualizable stable presentable \(\infty\)-categories (cf.~\cite{Efimov}).
I thank Peter Scholze
for valuable discussions related to this paper.
I thank
Deven Manam,
Juan~Esteban Rodríguez~Camargo,
and
Peter Scholze
for helpful comments on a draft of this paper.
I thank the Max Planck Institute for Mathematics.

This work was originally announced in~\cite{ttg-sm}
with the title “Very nuclear idempotents and continuous spectrum”,
as \(\Sm^{\con}\) was yet to be discovered
and the initial focus was on \(\Sm^{\rig}\).
I apologize for any confusion and for the delay in publication.

\subsection*{Conventions}\label{ss:conv}

We assume the reader is familiar with~\cite{ttg-sm}
and follow the same conventions.
For a presentable \(\infty\)-category~\(\cat{C}\),
we write \(\cat{C}_{\kappa}\) for the full subcategory
of \(\kappa\)-compact objects.

\section{Idealoids and restricted \texorpdfstring{\(\Ind\)}{Ind}-objects}\label{s:ind}

We introduce the notion of idealoids in \cref{ss:idl}.
Associated to them,
we get a special class of \(\Ind\)-objects,
which we study in \cref{ss:res}.

\subsection{Idealoids}\label{ss:idl}

\begin{definition}\label{xbjisd}
  An \emph{idealoid}
  of an \(\infty\)-category~\(\cat{C}\)
  is a class of morphisms~\(\cat{I}\)
  such that
  if \(f\in\cat{I}\),
  then \(g\circ f\circ h\in\cat{I}\)
  for any composable morphisms~\(g\) and~\(h\)
  in~\(\cat{C}\).
\end{definition}

\begin{remark}\label{xwii0l}
  Note that \cref{xbjisd}
  is \(1\)-categorical;
  an idealoid on~\(\cat{C}\) corresponds
  uniquely to an idealoid on its homotopy \(1\)-category.
\end{remark}

\begin{example}\label{x8lmc4}
  Let \(M\) be a monoid.
  Then we consider its classifying category~\(\cat{C}\),
  i.e.,
  the category with one object
  whose endomorphism monoid is~\(M\).
  In this situation,
  idealoids of~\(\cat{C}\) correspond
  to (two-sided) ideals of~\(M\).
  This is where the name in \cref{xbjisd} comes from;
  remember that “monoidoid” is another name for category.
\end{example}

\begin{example}\label{fun_idl}
Let \(\cat{I}\) be an idealoid of an \(\infty\)-category~\(\cat{C}\).
  For any \(\infty\)-category~\(\cat{K}\),
  we get the idealoid of \(\Fun(\cat{K},\cat{C})\)
  consisting of \(F\to G\)
  such that
  \(F(K)\to G(K)\) is in~\(\cat{I}\)
  for any \(K\in\cat{K}\).
  We write \(\Fun(K,\cat{I})\) for this idealoid.
\end{example}

We then consider special kinds of idealoids,
which play an important role in this paper:

\begin{definition}\label{subdiv}
  An idealoid~\(\cat{I}\) is \emph{subdivisible}
  if for any \(f\in\cat{I}\),
  there are \(g\) and \(h\in\cat{I}\)
  such that \(f=g\circ h\).
\end{definition}

We recall the following classical result:

\begin{theorem}[Cantor]\label{cantor}
  Let \(P\) be
  a totally ordered set
  of cardinality~\(\aleph_{0}\).
  Suppose that it is \emph{dense};
  i.e., for any \(a<b\) in~\(P\)
  there is \(c\in P\) satisfying \(a<c<b\).
  Then it is isomorphic to
  either
  \(\QQ\cap(0,1)\),
  \(\QQ\cap(0,1]\),
  \(\QQ\cap[0,1)\), or
  \(\QQ\cap[0,1]\).
\end{theorem}

\begin{example}\label{dyadic}
  By \cref{cantor},
  the poset
  \(\ZZ[1/2]\cap[0,1]\)
  is isomorphic to \(\QQ\cap[0,1]\).
  The so-called \emph{Minkowski question-mark function}
  gives a concrete isomorphism.
\end{example}

\begin{lemma}\label{x7sojd}
  Let \(\cat{I}\) be a subdivisible idealoid
  in an \(\infty\)-category~\(\cat{C}\).
  Then any morphism \(C(0)\to C(1)\) in~\(\cat{I}\)
  can be refined to a diagram
  \(C(\X)\colon\QQ\cap[0,1]\to\cat{C}\)
  such that
  \(C(a)\to C(b)\)
  is in~\(\cat{I}\) for any \(a<b\)
  in \(\QQ\cap[0,1]\).
\end{lemma}

\begin{proof}
  By induction,
  we construct
  a diagram \(C'(\X)\colon\ZZ[1/2]\cap[0,1]\to\cat{C}\)
  satisfying the same condition.
  We then use \cref{dyadic}
  to renumber it to obtain the desired diagram.
\end{proof}

\begin{remark}\label{x1hmh6}
  In the situation of \cref{x7sojd},
  suppose that \(\cat{C}\) has countable colimits.
  In that case,
  we can left Kan extend
  the diagram to obtain
  a diagram \([0,1]\to\cat{C}\)
  such that \(C(a)\to C(b)\) is in~\(\cat{I}\)
  for any \(a<b\).
\end{remark}

\begin{lemma}\label{xgpzzq}
  For any idealoid~\(\cat{I}\),
  there exists a largest subdivisible
  subidealoid of~\(\cat{I}\).
\end{lemma}

\begin{proof}
  Consider the class~\(\cat{J}\)
  consisting of all morphisms \(C(0) \to C(1)\)
  that can be refined to a diagram
  \(C(\X)\colon \QQ \cap [0,1] \to \cat{C}\)
  such that each morphism \(C(a) \to C(b)\)
  is in~\(\cat{I}\) for all \(a < b\).
  Then \(\cat{J}\) is a subidealoid of~\(\cat{I}\).
  It is subdivisible,
  since for any \(C(0) \to C(1) \in \cat{J}\),
  the refinement \(C(\X)\) yields
  \(C(0) \to C(1/2)\) and \(C(1/2) \to C(1) \in \cat{J}\).
  Finally, by \cref{x7sojd}, \(\cat{J}\) is the largest
  among subdivisible subidealoids of~\(\cat{I}\).
\end{proof}

\begin{definition}\label{xhyy8b}
  We write \(\cat{I}_{\sd}\) for the largest subidealoid
  of an idealoid~\(\cat{I}\)
  that is subdivisible,
  which exists by \cref{xgpzzq}.
\end{definition}

\begin{example}\label{fun_sd}
  We continue \cref{fun_idl}.

  Note that when \(\cat{I}\) is subdivisible,
  \(\Fun(\cat{K},\cat{I})\) need not be subdivisible.
  Take~\(\cat{C}\) to be the poset~\(Q\) given
  by \(P=P'=\QQ\cap[0,1]\) glued together
  at~\(0\in P\), \(0'\in P'\)
  and~\(1\in P\), \(1'\in P'\).
  Then the idealoid~\(\cat{I}\) consisting of nonidentity morphisms
  is subdivisible.
  Take \(\cat{K}=\Delta^{1}\)
  and consider the unique natural transformation
  from \(0=0'\to(1/2)'\) to \(1/2\to 1=1'\).
  Then this does not factor in the desired way.

  This also shows that
  \(\Fun(\cat{K},\cat{I})_{\sd}\subset\Fun(\cat{K},\cat{I}_{\sd})\)
  is not an equality in general.
\end{example}

However,
having a small buffer solves the problem
in \cref{fun_sd},
at least for \(\cat{K}=\Delta^{1}\):

\begin{example}\label{fun_buf}
Let \(\cat{I}\) be a subdivisible idealoid
  in an \(\infty\)-category~\(\cat{C}\).
  Suppose that we have a morphism
  \((C_{0}\to C'_{0})\to(C_{1}\to C'_{1})\)
  in \(\Fun(\Delta^{1},\cat{I})\).
  Assume that we have a factorization
  \begin{equation*}
    \begin{tikzcd}
      C_{0}\ar[rr]\ar[d]&
      {}&
      C_{1}\ar[dl]\ar[d]\\
      C'_{0}\ar[r]&
      C'_{1-\epsilon}\ar[r]&
      C'_{1}
    \end{tikzcd}
  \end{equation*}
  such that \(C'_{1-\epsilon}\to C'_{1}\)
  is in~\(\cat{I}\).
  We claim that
  the original morphism
  is in \(\Fun(\Delta^{1},\cat{I})_{\sd}\).
  In fact,
  the morphism
  \((C_{0}\to C'_{1-\epsilon})\to(C_{1}\to C'_{1})\)
  is in \(\Fun(\Delta^{1},\cat{I})_{\sd}\).
  The subdivision can be constructed
  by considering subdivisions
  of \(C_{0}\to C_{1}\)
  and \(C'_{1-\epsilon}\to C'_{1}\).
\end{example}

\begin{example}\label{xh40rh}
Subdivisibility also plays an important role
  in the following observation:
  Let \(\cat{C}_{0} \hookrightarrow \cat{C}\) be a full subcategory inclusion.
  Suppose \(\cat{I}\) is an ideal in~\(\cat{C}\),
  and let \(\cat{I}_{0}\) denote its pullback to~\(\cat{C}_{0}\).
  A natural question is whether \(\cat{I}\) is generated by~\(\cat{I}_{0}\).
  A necessary condition is that
  every morphism in~\(\cat{I}\)
  factors through an object of~\(\cat{C}_{0}\).
  In general, this is \emph{not} sufficient.
  However, if \(\cat{I}\) is subdivisible,
  this condition does imply that \(\cat{I}_{0}\) generates~\(\cat{I}\).
\end{example}

\subsection{Restricted \texorpdfstring{\(\Ind\)}{Ind}-objects}\label{ss:res}

Before
considering restricted \(\Ind\)-objects,
we first start with remarks
on general \(\Ind\)-objects.

\begin{remark}\label{xc7lf4}
  For a presentable \(\infty\)-category~\(\cat{C}\),
  we define
  \begin{equation*}
    \Ind(\cat{C})
    =
    \bigcup_{\kappa}\Ind(\cat{C}_{\kappa}),
  \end{equation*}
  where \(\kappa\) runs over infinite regular cardinals
  and \(\cat{C}_{\kappa}\) denotes the full subcategory
  of \(\kappa\)-compact objects.
  We call its objects
  \emph{\(\Ind\)-objects} of~\(\cat{C}\).

  When a presentable \(\infty\)-category~\(\cat{C}\)
  is small\footnote{This happens only when \(\cat{C}\) is essentially a poset.
  },
  the definition of \(\Ind(\cat{C})\) in \cref{xc7lf4}
  coincides with the usual definition.
\end{remark}

\begin{lemma}\label{xsoyii}
  Let \((C_{i})_{i}\to(D_{i})_{i}\) be a map
  between directed diagrams
  in an \(\infty\)-category~\(\cat{C}\).
  Suppose that for each~\(i\),
  there is \(j\geq i\)
  and a filler in the diagram
  \begin{equation*}
    \begin{tikzcd}
      C_{i}\ar[r]\ar[d]&
      C_{j}\ar[d]\\
      D_{i}\ar[r]\ar[ur,dashed]&
      D_{j}\rlap.
    \end{tikzcd}
  \end{equation*}
  Then it induces
  an equivalence between \(\Ind\)-objects.
\end{lemma}

\begin{proof}
  By Yoneda,
  we have to show that
  under the assumption
  it induces an equivalence
  between the colimits
  when \(\cat{C}=\Cat{S}\).
  By considering homotopy groups,
  we can furthermore
  replace~\(\Cat{S}\) with \(\Cat{Set}\).
  In that case,
  we can just check that
  it is bijective by hand.
\end{proof}

Now,
we define restricted \(\Ind\)-objects:

\begin{definition}\label{strict}
  Let \(\cat{I}\) be a class of morphisms
  of an \(\infty\)-category~\(\cat{C}\).
  We say that an \(\Ind\)-object~\(X\)
is \emph{\(\cat{I}\)-restricted}
  if any morphism \(C\to X\)
  from \(C\in\cat{C}\)
  can be factored as
  \(C\to C'\to X\)
  with the first map being in~\(\cat{I}\).
  We write
  \(\Ind_{\cat{I}}(\cat{C})\)
  for the full subcategory
  of \(\Ind(\cat{C})\)
  spanned by
  \(\cat{I}\)-restricted \(\Ind\)-objects.
\end{definition}

This notion behaves
particularly nicely
when \(\cat{I}\) is an idealoid:

\begin{lemma}\label{x5hdob}
  Let \(\cat{I}\) be an idealoid
  in an \(\infty\)-category~\(\cat{C}\).
  An \(\Ind\)-object~\(X\)
  is \(\cat{I}\)-restricted
  if and only if
  it is equivalent to
  the colimit of a directed diagram
  \(D\colon P\to\cat{C}\)
  where
  for any~\(p\)
  there is a~\(p'\geq p\)
  such that \(D(p)\to D(p')\) is in~\(\cat{I}\).
\end{lemma}

\begin{proof}
  We prove the “if” direction.
  Let \(C\to X\) be a morphism.
  Then there is \(p\in P\) such that
  it decomposes as
  \(C\to D(p)\to X\).
  Then we can take \(p'\geq p\)
  such that \(D(p)\to D(p')\) is in~\(\cat{I}\).
  Then \(C\to D(p')\to X\) is the desired factorization.

  We then prove the “only if” direction.
  Let \(X\) be a \(\cat{I}\)-restricted \(\Ind\)-object.
  We write it as a colimit
  of a filtered diagram
  \(D\colon P\to\cat{C}\).
  We claim that this satisfies the condition.
  Indeed, for \(p\in P\), by condition,
  \(D(p)\to X\)
  factors as \(D(p)\to C\to X\)
  where \(D(p)\to C\) is in~\(\cat{I}\).
  We can then take
  \(p'\geq p\)
  such that \(C\to X\) factors as
  \(C\to D(p')\to X\).
  Now \(D(p)\to D(p')\) is in~\(\cat{I}\).
\end{proof}

\begin{lemma}\label{xm2a0p}
  Let \(\cat{I}\) be an idealoid
  in an \(\infty\)-category~\(\cat{C}\).
  Then
  the full subcategory
  \(\Ind_{\cat{I}}(\cat{C})\subset\Ind(\cat{C})\)
  is closed under filtered colimits.
\end{lemma}

\begin{proof}
  Let \(X=\injlim_{i}X_{i}\)
  be a filtered colimit
  of \(\cat{I}\)-restricted \(\Ind\)-objects.
  Then any \(C\to X\)
  from \(C\in\cat{C}\)
  factors through some~\(X_{i}\),
  and therefore it factors through
  a morphism \(C\to C'\) in~\(\cat{I}\).
\end{proof}

\begin{remark}\label{xxvm0b}
  In the situation of \cref{xm2a0p},
  the full subcategory need not be closed under all colimits.
  For example,
  when \(\cat{I}\) is empty,
  \(\Ind_{\cat{I}}(\cat{C})=\emptyset\)
  does not admit an initial object.
\end{remark}

\begin{lemma}\label{sequential}
  The full subcategory
  \(\Ind_{\cat{I}}(\cat{C})\)
  is generated under filtered colimits
  by objects
  \((C_{n})_{n}\)
  such that \(C_{n}\to C_{n+1}\) is in~\(\cat{I}\).
\end{lemma}

\begin{proof}
  Let \(D\colon P\to\cat{C}\)
  be an \(\cat{I}\)-restricted \(\Ind\)-object
  indexed by a directed poset~\(P\).
  Let \(I\) be its idealoid
  defined as the pullback of the idealoid~\(\cat{I}\);
  i.e., \(p\leq q\) is in~\(I\)
  if and only if \(D(p)\to D(q)\) is in~\(\cat{I}\).
  Since \(D\) is \(\cat{I}\)-restricted,
  for any \(p\in P\),
  there is \(q\geq p\) such that \(p\leq q\) is in~\(I\).

We write \(\Seq(P)\)
  for the functor poset \(\Fun(\NN,P)\).
  We write \(\Seq_{I}(P)\)
  for its full subposet
  spanned by the sequences \((p_{n})_{n}\)
  such that \(p_{n}\leq p_{n+1}\) is in~\(I\).
  We claim that \(\Seq_{I}(P)\) is directed.
  First,
  since \(P\) is directed,
  we have an element \(p\in P\).
  Starting from this,
  we get a sequence in \(\Seq_{I}(P)\).
  Therefore, it is nonempty.
  Now, let \((p_{n})_{n}\) and \((p'_{n})_{n}\)
  be two elements of \(\Seq_{I}(P)\).
  Since \(P\) is directed,
  we can take \(q_{0}\)
  that is greater than or equal to~\(p_{0}\) and~\(p'_{0}\).
  Then we take \(q_{0}\leq r_{0}\) in~\(I\).
  Since \(P\) is directed,
  we can take~\(q_{1}\) that is greater than or equal to
  \(p_{1}\), \(p'_{1}\), and~\(r_{0}\).
  By repeating this process,
  we get a sequence \((q_{n})_{n}\)
  that is greater than or equal to~\((p_{n})_{n}\) and~\((p'_{n})_{n}\).

  We have the composite
  \begin{equation*}
    f\colon
    \Seq_{I}(P)\times\NN
    \hookrightarrow\Seq(P)\times\NN
    \to P,
  \end{equation*}
  where the second map is given by \(((p_{n})_{n},n)\mapsto p_{n}\).
  The cofinality of~\(f\)
  follows from the argument for the nonemptiness above,
  since we can start from arbitrary \(p_{0}\).
  This shows that
  \(D\circ f\) computes the same colimit.
  We can first take the colimit in~\(\NN\)
  to see the desired result.
\end{proof}

We illustrate the utility of \cref{sequential}
through the following examples:

\begin{example}\label{fun_ind}
  We continue \cref{fun_idl}.
  Note from~\cite[Proposition~5.3.5.15]{LurieHTT}
  that
  when \(\cat{K}=P\) is a finite poset,
  then the functor \(\Ind(\Fun(P,\cat{C}))\to\Fun(P,\Ind(\cat{C}))\)
  is an equivalence.
  We claim that
  the induced functor
  \(\Ind_{\Fun(P,\cat{I})}(\Fun(P,\cat{C}))\to\Fun(P,\Ind_{\cat{I}}(\cat{C}))\)
  is also an equivalence.
  By \cref{sequential},
  it suffices to show that
  \(P\)-indexed sequential \(\cat{I}\)-restricted \(\Ind\)-objects
  are inside the image.

  We observe this for \(P=\Delta^{1}\),
  which is the case we need.
  We consider \((C_{n})_{n}\to(C'_{n})_{n}\)
  in \(\Ind_{\cat{I}}(\cat{C})\)
  such that \(C_{n}\to C_{n+1}\) and \(C'_{n}\to C'_{n+1}\)
  are in~\(\cat{I}\).
  We consider \(C_{0}\to\injlim_{n}C'_{n}\).
  Then by definition,
  it factors through \(C'_{f(0)}\)
  for some~\(f(0)\).
  We then consider
  \(C_{1}\to\injlim_{n>f(0)}C'_{n}\),
  but it factors through \(C'_{f(1)}\)
  for some \(f(1)>f(0)\).
  Repeating this process,
  we obtain a map \(f\colon\NN\to\NN\)
  such that \(f(n+1)>f(n)\) and
  \(C_{n}\to\injlim_{n}C'_{n}\)
  factors through \(C_{n}\to C'_{f(n)}\).
  Since \(f\) is cofinal,
  the original morphism is the image of
  \((C_{n}\to C'_{f(n)})_{n}\).
  The general case follows from a simple induction,
  which we leave to the reader.
\end{example}

\begin{example}\label{mini}
  Let \(\cat{C}\) be a \(\kappa\)-compactly generated \(\infty\)-category.
  Then the class of morphisms
  that factor through a \(\kappa\)-compact object
  is an idealoid~\(\cat{I}\).
  With this,
  the trivial inclusion
  \(\Ind(\cat{C}_{\kappa})\hookrightarrow\Ind_{\cat{I}}(\cat{C})\)
  is an equivalence.
  This can be seen using \cref{sequential}.
  Indeed,
  we need to see that
  the \(\Ind\)-object
  \((C_{n})_{n}\) such that \(C_{n}\to C_{n+1}\) is in~\(\cat{I}\)
  is in the essential image,
  which has an equivalent
  sequential \(\Ind\)-object in \(\cat{C}_{\kappa}\).
\end{example}

\section{Compact morphisms and continuous categories}\label{s:con}

The notion of a continuous \(1\)-category
was introduced by Johnstone–Joyal~\cite{JohnstoneJoyal82};
see also~\cite[Section~21.1.2]{LurieSAG},~\cite[Section~4.2]{AnelLejay},
and~\cite[Section~2]{Ramzi1}
for \(\infty\)-categorical accounts.
In this section,
we review some aspects of the theory
of continuous \(\infty\)-categories (see \cref{xv8wwv})
and introduce new terminology.
Most importantly,
we recall the following well-known facts:

\begin{proposition}\label{con_0}
  A presentable \(\infty\)-category is continuous
  if and only if every object
  is the colimit of some Schwartz \(\Ind\)-object
  (see \cref{xxc7v8}).
\end{proposition}

\begin{proposition}\label{con_1}
  Consider
  a colimit-preserving functor
  between continuous presentable \(\infty\)-categories.
  Its right adjoint
  preserves filtered colimits
  if and only if it preserves compact morphisms.
\end{proposition}

\begin{remark}\label{xcqr0g}
  More generally,
  one can consider the situation without finite colimits,
  but here we restrict ourselves to the presentable case.
\end{remark}

\begin{remark}\label{xah8cd}
  In the literature,
  one often considers a variant of \cref{con_0} where only
  “basic” (i.e., sequential) Schwartz objects are considered
  and the \(\infty\)-category is then generated under
  colimits by them.
  To deduce such a version,
  one can use \cref{sequential},
  our general combinatorial result
  that reduces general \(\Ind\)-objects
  to sequential ones.
\end{remark}

\subsection{Compact morphisms}\label{ss:wav}

\begin{definition}\label{xl2h0i}
  Let \(\cat{C}\) be a presentable \(\infty\)-category.
  A morphism \(C\to C'\) is called \emph{compact}
  if for any filtered diagram
  \((D_{j})_{j}\),
  there is a filler
  \begin{equation*}
    \begin{tikzcd}
      \injlim_{j}\Map(C',D_{j})\ar[r]\ar[d]&
      \injlim_{j}\Map(C,D_{j})\ar[d]\\
      \Map\bigl(C',\injlim_{j}D_{j}\bigr)
      \ar[r]\ar[ur,dashed]&
      \Map\bigl(C,\injlim_{j}D_{j}\bigr)
    \end{tikzcd}
  \end{equation*}
  in~\(\Cat{S}\).
  Note that
  the condition depends only on
  \((D_{j})_{j}\) as an \(\Ind\)-object (see \cref{xc7lf4}).
\end{definition}

\begin{remark}\label{xo7dpa}
  The terminology in \cref{xl2h0i} may be confusing,
  since it could refer to compact objects in the \(\infty\)-category
  of morphisms,
  which are precisely those morphisms whose source and target are compact.
\end{remark}

\begin{example}\label{xq60rx}
  In \cref{xl2h0i},
  for an object~\(C\),
  the morphism \({\id}_{C}\) is compact
  if and only if \(C\) is compact.
\end{example}

\begin{example}\label{xo7ubc}
  In the situation of \cref{xl2h0i},
  suppose further that \(\cat{C}\) is compactly generated.
  In this case,
  \(C\to C'\) is compact if and only if
  it factors through a compact object,
  as we can see by expressing~\(C'\) as
  a filtered colimit of compact objects.
\end{example}

\begin{example}\label{xaps0r}
  In the situation of \cref{xl2h0i},
  suppose further that \(\cat{C}\) is
  \(\kappa\)-compactly generated
  for an infinite regular cardinal~\(\kappa\).
  Then still the argument in \cref{xo7ubc}
  shows that any compact morphism
  factors through a \(\kappa\)-compact object.
\end{example}

\begin{example}\label{way_below}
  In a complete poset (or more generally a poset with directed joins),
  we say that \(p\) is \emph{way below}~\(q\)
  and write \(p\ll q\)
  if \(p\leq q\) is compact in the sense of \cref{xl2h0i}.
  Concretely,
  this means that for any directed subset~\(D\),
  if \(q\leq\bigvee D\),
  there is \(d\in D\) such that \(p\leq d\).
\end{example}

\begin{example}\label{xa4f06}
In the situation of \cref{xl2h0i},
  the full subcategory
  of \(\Fun(\Delta^{1},\cat{C})\)
  spanned by compact morphisms
  is closed under finite coproducts.
  It is \emph{not} closed under finite colimits in general;
  e.g., in the stable case,
  any morphism is the cofiber
  of some morphism between compact morphisms.
  However, see \cref{xbagrg} for a positive result.
\end{example}

\begin{lemma}\label{cpt_idl}
  In the situation of \cref{xl2h0i},
  compact morphisms form an idealoid.
\end{lemma}

\begin{proof}
  This directly follows from the definition.
\end{proof}

\begin{lemma}\label{xbagrg}
In the situation of \cref{xl2h0i},
  consider a commutative diagram
  \begin{equation}
    \label{e:vqifd}
    \begin{tikzcd}
      C_{0}'\ar[r]&
      C_{1}'\ar[r]&
      C_{2}'\\
      C_{0}\ar[r]\ar[d]\ar[u]&
      C_{1}\ar[r]\ar[d]\ar[u]&
      C_{2}\ar[d]\ar[u]\\
      C_{0}''\ar[r]&
      C_{1}''\ar[r]&
      C_{2}''\rlap.
    \end{tikzcd}
  \end{equation}
  When \(C_{0}\to C_{1}\),
  \(C_{1}'\to C_{2}'\),
  and \(C_{1}''\to C_{2}''\)
  are compact,
  then the induced morphism
  \(\injlim(C_{0}'\gets C_{0}\to C_{0}'')
  \to\injlim(C_{2}'\gets C_{2}\to C_{2}'')\)
  is compact.
\end{lemma}

\begin{proof}
Consider a filtered colimit \(D=\injlim_{j}D_{j}\).
  We fix witnesses
  \(\Map(C_{1},D)\to\injlim_{j}\Map(C_{0},D_{j})\),
  \(\Map(C_{2}',D)\to\injlim_{j}\Map(C_{1}',D_{j})\),
  and \(\Map(C_{2}'',D)\to\injlim_{j}\Map(C_{1}'',D_{j})\).
  With them, we obtain the commutative diagram
  \begin{equation*}
    \begin{tikzcd}
      \Map(C_{2}',D)\ar[r]\ar[d]&
      \injlim_{j}\Map(C_{1}',D_{j})\ar[r]\ar[d]&
      \injlim_{j}\Map(C_{0}',D_{j})\ar[d]\\
      \Map(C_{2},D)\ar[r]&
      \Map(C_{1},D)\ar[r]&
      \injlim_{j}\Map(C_{0},D_{j})\\
      \Map(C_{2}'',D)\ar[r]\ar[u]&
      \injlim_{j}\Map(C_{1}'',D_{j})\ar[r]\ar[u]&
      \injlim_{j}\Map(C_{0}'',D_{j})\rlap.\ar[u]
    \end{tikzcd}
  \end{equation*}
  We get the desired witness
  by taking the limits of the left and right columns.
\end{proof}

\subsection{(Very) Schwartz \texorpdfstring{\(\Ind\)}{Ind}-objects}\label{ss:sch}

\begin{definition}\label{xxc7v8}
  Let \(\cat{C}\) be a presentable \(\infty\)-category.
  We write~\(\cat{I}\) for
  the class of compact morphisms—this is an idealoid by \cref{cpt_idl}.
  We call a morphism in \(\cat{I}_{\sd}\) (see \cref{xhyy8b})
  a \emph{very compact morphism}.
  We call an object of \(\Ind_{\cat{I}}(\cat{C})\)
  a \emph{Schwartz \(\Ind\)-object}
  and an object of \(\Ind_{\cat{I}_{\sd}}(\cat{C})\)
  a \emph{very Schwartz \(\Ind\)-object}.
\end{definition}

\begin{remark}\label{xzsrxi}
  Grothendieck introduced the notion of Schwartz spaces
in~\cite[Section~III.4]{Grothendieck54}.
The name in \cref{xxc7v8} comes from this,
  but see \cref{x17h0a} below.
\end{remark}

\begin{remark}\label{x17h0a}
  A nuclear module
  in condensed mathematics
  (see \cref{xsdc04})
  captures the idea of \emph{dual} nuclear space
  in functional analysis.
  The usage of the adjective “Schwartz” here
  follows this convention:
  Schwartz spaces in functional analysis (see \cref{xzsrxi})
  are typically limits—rather than colimits—along compact transitions\footnote{Of course,
    this compactness is different from the compactness in \cref{xl2h0i}.
  }.
  In our setting,
  limits play a minor role, so we drop “dual”.
\end{remark}

\begin{remark}\label{xhfgld}
  In \cref{xxc7v8},
  we consider
  \(\Ind\) of a presentable \(\infty\)-category
  (see \cref{xc7lf4}).
  However,
  this is not necessary.
  When \(\cat{C}\)
  is \(\kappa\)-compactly generated
  for an infinite regular cardinal~\(\kappa\),
  \cref{mini,xaps0r} show that
  any Schwartz \(\Ind\)-object
  is in \(\Ind(\cat{C}_{\kappa})\).
\end{remark}

\begin{example}\label{xa4ihu}
  A constant \(\Ind\)-object is
  Schwartz if and only if it is compact.
\end{example}

\begin{lemma}\label{xd0jpr}
  Consider a presentable \(\infty\)-category
  and a Schwartz \(\Ind\)-object~\((C_{i})_{i}\).
  Then the morphism
  \begin{equation*}
    \Map_{\Ind(\cat{C})}((C_{i})_{i},(D_{j})_{j})
    \to
    \Map_{\cat{C}}\biggl(\injlim_{i}C_{i},\injlim_{j}D_{j}\biggr)
  \end{equation*}
  is an equivalence
  for any \(\Ind\)-object~\((D_{j})_{j}\).
\end{lemma}

\begin{proof}
  The morphism is equivalent to
  \begin{equation*}
    \projlim_{i}
    \injlim_{j}
    \Map(C_{i},D_{j})
    \to
    \projlim_{i}
    \Map\biggl(C_{i},\injlim_{j}D_{j}\biggr).
  \end{equation*}
  By the definition of Schwartz \(\Ind\)-object
  and \cref{xsoyii},
  \begin{equation*}
    \biggl(\injlim_{j}\Map(C_{i},D_{j})\biggr)_{i}
    \to
    \biggl(\Map\biggl(C_{i},\injlim_{j}D_{j}\biggr)\biggr)_{i}
  \end{equation*}
  is an equivalence of \(\Pro\)-spaces.
  Hence the desired claim follows.
\end{proof}

\begin{corollary}\label{unique}
  Let \(\cat{I}\) be the idealoid of compact morphisms.
  Then the colimit functor
  \(\Ind_{\cat{I}}(\cat{C})\to\cat{C}\)
  is fully faithful.
\end{corollary}

\subsection{Continuous categories}\label{ss:con}

\begin{definition}\label{xaq9jc}
  We write \(\Cat{Pr}^{\cg}\)
  for the subcategory of~\(\Cat{Pr}\)
  spanned by compactly generated \(\infty\)-categories
  and functors preserving colimits and compact objects.
\end{definition}

\begin{definition}\label{xv8wwv}
  A presentable \(\infty\)-category
  is \emph{continuous} if
  it is a retract of
  a compactly generated \(\infty\)-category
  in~\(\Cat{Pr}\).
  Similarly,
  a morphism of continuous presentable \(\infty\)-categories
  is a retract
  of a morphism in \(\Cat{Pr}^{\cg}\)
  in \(\Fun(\Delta^{1},\Cat{Pr})\).
  We write \(\Cat{Pr}^{\con}\)
  for this subcategory of~\(\Cat{Pr}\).
\end{definition}

\begin{example}\label{x9qljx}
  Compactly generated \(\infty\)-categories
  are continuous.
  Moreover,
  it follows from \cref{xhfb2c} below that
  the functor \(\Cat{Pr}^{\cg}\to\Cat{Pr}^{\con}\)
  is a full inclusion.
\end{example}

We have a canonical (up to the choice of a regular cardinal)
way to present a continuous presentable \(\infty\)-category
as a retract of a compactly generated one:

\begin{lemma}\label{x9m28r}
  Let \(\cat{C}\)
  be a presentable \(\infty\)-category.
  It is continuous
  if and only if
  for a sufficiently large regular cardinal~\(\kappa\)
  the functor
  \(G\colon\Ind(\cat{C}_{\kappa})\to\cat{C}\)
  obtained as the left Kan extension of the inclusion
  \(\cat{C}_{\kappa}\hookrightarrow\cat{C}\)
  has a left adjoint.
\end{lemma}

\begin{proof}
  The “if” direction is clear.
  We prove the “only if” direction.

  We first consider the case
  when \(\cat{C}\) is compactly generated.
  We write \(\cat{A}\) for the full subcategory
  of compact objects.
  We prove that \(\cat{C}\simeq\Ind(\cat{A})\) satisfies the property
  for any~\(\kappa\)
  by showing that
  the functor
  \(F\colon\Ind(\cat{A})\to\Ind(\Ind(\cat{A})_{\kappa})\)
  obtained as
  the left Kan extension
  of the inclusion
  is left adjoint to~\(G\).
  To see this,
  it suffices to construct
  a functorial equivalence
  \begin{equation*}
    \Map_{\Ind(\Ind(\cat{A})_{\kappa})}(F(C),D)
    \simeq
    \Map_{\Ind(\cat{A})}(C,G(D)).
  \end{equation*}
  First, it suffices to
  consider the case \(C\in\cat{A}\).
  In that case,
  since \(C\) and \(F(C)\) are both compact,
  it suffices to consider the case
  \(D\in\Ind(\cat{A})_{\kappa}\).
  Then we have an obvious choice of
  such an equivalence.

We now consider the general case.
  We take a regular cardinal~\(\kappa\)
  such that
  \(\cat{C}\) is a retract
  in \(\Cat{Pr}^{\kappa}\)
  of a compactly generated \(\infty\)-category~\(\cat{D}\).
  Then \(G\) is a retract
  of the functor
  \(\Ind(\cat{D}_{\kappa})\to\cat{D}\)
  in \(\Cat{Pr}\).
  But the previous argument
  shows that this has a left adjoint.
  Therefore, so does~\(G\);
  see, e.g.,~\cite[Lemma~21.1.2.14]{LurieSAG}.
\end{proof}

\begin{lemma}\label{x11iiz}
  Let \(F\colon\cat{C}\to\cat{D}\) be
  a fully faithful functor between presentable \(\infty\)-categories
  with right adjoint~\(G\).
  Suppose that \(G\) preserves filtered colimits.
  Then a morphism \(C\to C'\) in \(\cat{C}\)
  is compact if and only if \(FC\to FC'\) is compact.
\end{lemma}

\begin{proof}
  We prove the “if” direction
  by checking
  any filtered diagram \((D_{j})_{j}\) in~\(\cat{C}\)
  satisfies the condition in \cref{xl2h0i}.
  We see this
  by using the compactness of \(FC\to FC'\)
  against the diagram \((FD_{j})_{j}\).

  We then prove the “only if” direction
  by constructing a filler
  \begin{equation*}
    \begin{tikzcd}
      \injlim_{j}\Map(FC',D_{j})\ar[r]\ar[d]&
      \injlim_{j}\Map(FC,D_{j})\ar[d]\\
      \Map\bigl(FC',\injlim_{j}D_{j}\bigr)
      \ar[r]\ar[ur,dashed]&
      \Map\bigl(FC,\injlim_{j}D_{j}\bigr)
    \end{tikzcd}
  \end{equation*}
  for any filtered diagram \((D_{j})_{j}\) in~\(\cat{D}\).
  We see this
  by using the compactness of \(C\to C'\)
  against the diagram \((GD_{j})_{j}\)
\end{proof}

\begin{proof}[Proof of \cref{con_0}]
  We first prove the “if” direction.
  Suppose that \(\cat{C}\) is a presentable \(\infty\)-category
  such that every object is the colimit of
  some Schwartz \(\Ind\)-object.
  We fix a regular cardinal~\(\kappa\)
  such that \(\cat{C}\) is \(\kappa\)-compactly generated.
  By \cref{x9m28r},
  it suffices to prove that
  the functor
  \(\Ind(\cat{C}_{\kappa})\to\cat{C}\)
  obtained as the left Kan extension of the inclusion
  \(\cat{C}_{\kappa}\hookrightarrow\cat{C}\)
  admits a left adjoint.
To see this,
  for each object~\(C\),
  we need to find an \(\Ind\)-object \((C_{i})_{i}\)
  of~\(\cat{C}_{\kappa}\)
  with colimit~\(C\)
  such that
  \begin{equation*}
    \Map_{\Ind(\cat{C})}((C_{i})_{i},(D_{j})_{j})
    \to
    \Map_{\cat{C}}\biggl(\injlim_{i}C_{i},\injlim_{j}D_{j}\biggr)
  \end{equation*}
  is an equivalence
  for any \(\Ind\)-object \((D_{j})_{j}\) of \(\cat{C}_{\kappa}\).
  By assumption,
  we can take a Schwartz \(\Ind\)-object~\(X\)
  whose colimit is~\(C\).
  Let \((C_{i})_{i}\)
  be the image of~\(X\)
  under the functor
  \(\Ind(\cat{C})\to\Ind(\cat{C}_{\kappa})\)
  that is obtained as the left Kan extension of
  \(\cat{C}\to\Ind(\cat{C}_{\kappa})\),
  which is the left adjoint of the functor
  \(\Ind(\cat{C}_{\kappa})\to\cat{C}\).
  Then this satisfies the desired property by \cref{xd0jpr}.

  We then prove the converse.
  Consider an object~\(D\) in a continuous presentable \(\infty\)-category~\(\cat{C}\).
  By \cref{x9m28r},
  we can take \(\kappa\) such that
  the Yoneda embedding \(j\colon\cat{C}\to\Ind(\cat{C}_{\kappa})\)
  admits a twice left adjoint~\(\check{\jmath}\).
  We claim that \(\check{\jmath}(D)\)
  is a Schwartz \(\Ind\)-object of~\(\cat{C}_{\kappa}\)\footnote{To be precise,
    here we mean that
    it is in \(\Ind_{\cat{I}}(\cat{C}_{\kappa})\),
    where \(\cat{I}\) is the class of compact morphisms of~\(\cat{C}\)
    that are in~\(\cat{C}_{\kappa}\).
  }.
First,
  we write
  \(\check{\jmath}(D)\)
  as a filtered colimit \(\injlim_{j}j(D_{j})\),
  where \(D_{j}\in\cat{C}_{\kappa}\).
  Since \(\check{\jmath}\) preserves colimits,
  the morphism
  \(\injlim_{j}\check{\jmath}(D_{j})\to\injlim_{j}j(D_{j})\)
  is an equivalence.
  Then we consider \(C\in\cat{C}_{\kappa}\)
  and a morphism
  \(j(C)\to\check{\jmath}(D)\).
  By the observation above,
  it factors through
  \(\check{\jmath}(D_{j})\to j(D_{j})\)
  for some~\(j\).
  Applying the same argument again,
  we can factor
  the map \(j(D_{j})\to\check{\jmath}(D)\)
  through
  \(\check{\jmath}(D_{k})\to j(D_{k})\)
  for some \(k\geq j\).
  Then the map \(C\to D_{k}\) is compact by \cref{x11iiz},
  and therefore \(\check{\jmath}(D)\) is Schwartz.
\end{proof}

The proof above shows the following:

\begin{corollary}\label{xrz98z}
  In a continuous presentable \(\infty\)-category,
  the idealoid of compact morphisms
  is subdivisible;
  i.e., compact morphisms are very compact.
\end{corollary}

\begin{proof}
  We use the notation in the proof of \cref{con_0} above.
  Let \(C\to C'\) be a compact morphism in~\(\cat{C}\).
  Then the morphism \(\check{\jmath}(C)\to\check{\jmath}(C')\)
  in \(\Ind(\cat{C}_{\kappa})\) is compact
  by \cref{x11iiz}.
  By \cref{xo7ubc},
  it factors through
  \(j(D)\to\check{\jmath}(C')\) for some \(D\in\cat{C}_{\kappa}\).
  As in the proof of \cref{con_0},
  this map factors through
  \(\check{\jmath}(D')\to j(D')\).
  Therefore, the morphism
  \(\check{\jmath}(C)\to\check{\jmath}(C')\)
  factors
  into
  \(\check{\jmath}(C)\to j(D)\to\check{\jmath}(D')\to j(D')\to\check{\jmath}(C')\).
  By \cref{x11iiz},
  the morphisms \(C\to D'\) and \(D'\to C'\) are compact.
\end{proof}

We then consider morphisms in \(\Cat{Pr}^{\con}\).
The definition
can be also justified using this adjunction:

\begin{lemma}\label{mor_adj}
  For a colimit-preserving functor
  between continuous presentable \(\infty\)-categories
  \(\cat{C}\to\cat{D}\),
  it is a morphism in \(\Cat{Pr}^{\con}\)
  if and only if
  \begin{equation}
    \label{e:bc7qh}
    \begin{tikzcd}
      \Ind(\cat{C}_{\kappa})\ar[r]\ar[d]&
      \cat{C}\ar[d]\\
      \Ind(\cat{D}_{\kappa})\ar[r]&
      \cat{D}
    \end{tikzcd}
  \end{equation}
  is (horizontally) left adjointable
  for a sufficiently large regular cardinal~\(\kappa\).
\end{lemma}

\begin{proof}
  The “if” direction is clear.
  For the “only if” direction,
  we can argue as in the proof of \cref{x9m28r}.
  In the compactly generated case,
  the desired adjointability
  follows from the description of~\(F\) in \cref{x9m28r}.
  The general case follows from that case.
\end{proof}

\begin{lemma}\label{xhfb2c}
  A colimit-preserving functor between
  continuous presentable \(\infty\)-categories
  is a morphism in \(\Cat{Pr}^{\con}\)
  if and only if its right adjoint preserves filtered colimits.
\end{lemma}

\begin{proof}
  The “only if” direction directly follows from the definition.
  For the “if” direction,
  we use \cref{mor_adj}.
  Since the square is horizontally left adjointable
  if and only if it is vertically right adjointable,
  it suffices to observe that the natural transformation
  from
  \(\Ind(\cat{D}_{\kappa})\to\Ind(\cat{C}_{\kappa})\to\cat{C}\)
  to
  \(\Ind(\cat{D}_{\kappa})\to\cat{D}\to\cat{C}\)
  is an equivalence.
  Since all the functors preserve filtered colimits,
  it suffices to check this for an object of~\(\cat{D}_{\kappa}\),
  which is clear.
\end{proof}

\begin{proof}[Proof of \cref{con_1}]
  By \cref{xhfb2c},
  it suffices to show that a colimit-preserving functor
  \(F\colon\cat{C}\to\cat{D}\) between continuous presentable \(\infty\)-categories
  is a morphism in \(\Cat{Pr}^{\con}\)
  if and only if it preserves compact morphisms.

  We prove the “only if” direction.
  Let \(C\to C'\) be a compact morphism.
  We take \(\kappa\) sufficiently large so that
  \(\cat{C}\to\cat{D}\) is in \(\Cat{Pr}^{\kappa}\) and
  \(C\) and~\(C'\) are \(\kappa\)-compact.
  By \cref{mor_adj},
  the diagram
  \begin{equation*}
    \begin{tikzcd}
      \cat{C}\ar[r,"\check{\jmath}"]\ar[d]&
      \Ind(\cat{C}_{\kappa})\ar[d]\\
      \cat{D}\ar[r,"\check{\jmath}"]&
      \Ind(\cat{D}_{\kappa})
    \end{tikzcd}
  \end{equation*}
  commutes.
  By \cref{x11iiz},
  we see that \(\check{\jmath}(C)\to\check{\jmath}(C')\)
  is compact.
  Hence by \cref{xo7ubc},
  it factors through \(j(C'')\) for some \(C''\in\cat{C}_{\kappa}\).
  Therefore,
  the morphism
  \(\check{\jmath}(FC)\to\check{\jmath}(FC')\) also
  factors through \(j(FC'')\).
  By \cref{x11iiz}, we see that \(FC\to FC'\) is compact.

  For the “if” direction, assume that \(F\) preserves compact morphisms.
  We must show that \cref{e:bc7qh} is left adjointable.
  This follows from the fact—observed in the proof of \cref{con_0}—that
  the functor \(\check{\jmath}\) presents an object
  as a Schwartz \(\Ind\)-object, uniquely so by \cref{unique}.
\end{proof}

\subsection{Tensor products of continuous categories}\label{ss:smon}

\begin{proposition}\label{xwccpi}
  The symmetric monoidal structure on~\(\Cat{Pr}\)
  constructed in~\cite[Section~4.8.1]{LurieHA}
  restricts to~\(\Cat{Pr}^{\con}\).
\end{proposition}

\begin{proof}
  It suffices to check
  that \(\cat{C}\otimes\cat{D}\to\cat{C}\otimes\cat{D}'\)
  is in \(\Cat{Pr}^{\con}\)
  for an object~\(\cat{C}\)
  and a morphism~\(\cat{D}\to\cat{D}'\) in \(\Cat{Pr}^{\con}\).
  This formally follows from \cref{xv8wwv}
  and the fact that
  the symmetric monoidal structure on~\(\Cat{Pr}\)
  restricts to~\(\Cat{Pr}^{\cg}\).
\end{proof}

We can describe
the binary tensor product
concretely as follows:

\begin{proposition}\label{xlcwpp}
  Consider \(\cat{C}\), \(\cat{D}\), and~\(\cat{E}\in\Cat{Pr}^{\con}\)
  and a functor \(F\colon\cat{C}\times\cat{D}\to\cat{E}\)
  that preserves colimits in each variable.
  Then the corresponding morphism
  \(\cat{C}\otimes\cat{D}\to\cat{E}\) in~\(\Cat{Pr}\)
  is in \(\Cat{Pr}^{\con}\) if and only if
  for any compact morphisms \(C\to C'\) in~\(\cat{C}\)
  and \(D\to D'\) in~\(\cat{D}\),
  the morphism \(F(C,D)\to F(C',D')\) is compact.
\end{proposition}

\begin{proof}
  To prove the “only if” direction,
  it suffices to prove that
  \(C\boxtimes D\to C'\boxtimes D'\) is compact
  for any compact morphisms \(C\to C'\) and \(D\to D'\).
  This follows by writing~\(\cat{C}\) and~\(\cat{D}\)
  as in \cref{x9m28r}
  and applying \cref{x11iiz}.

  We then prove the “if” direction
  by showing that
  the right adjoint
  \(G\colon\cat{E}\to\cat{C}\otimes\cat{D}\)
  preserves filtered colimits.
  Let \(E=\injlim_{k}E_{k}\)
  be a filtered colimit in~\(\cat{E}\).
  We wish to show that
  \begin{equation*}
    \Map_{\cat{C}\otimes\cat{D}}\biggl(C\boxtimes D,\injlim_{k}G(E_{k})\biggr)
    \to
    \Map_{\cat{C}\otimes\cat{D}}(C\boxtimes D,G(E))
    \simeq
    \Map_{\cat{E}}(F(C,D),E)
  \end{equation*}
  is an equivalence
  for any \(C\in\cat{C}\) and \(D\in\cat{D}\).
  We write~\(C\) and~\(D\)
  as the colimits
  of very Schwartz \(\Ind\)-objects.
  By considering the product of the index \(\infty\)-categories,
  we obtain presentations~\((C_{i})_{i}\) and~\((D_{i})_{i}\)
  indexed by the same \(\infty\)-category.
  By the argument above,
  \((F(C_{i},D_{i}))_{i}\)
  is again very Schwartz.
  We then use \cref{xd0jpr} twice to obtain
  \begin{align*}
    \Map_{\cat{C}\otimes\cat{D}}\biggl(C\boxtimes D,\injlim_{k}G(E_{k})\biggr)
&\simeq
    \Map_{\Ind(\cat{C}\otimes\cat{D})}((C_{i}\boxtimes D_{i})_{i},(G(E_{k}))_{k})\\
    &\simeq
    \Map_{\Ind(\cat{E})}((F(C_{i},D_{i}))_{i},(E_{k})_{k})
    \simeq
    \Map_{\cat{E}}(F(C,D),E),
  \end{align*}
  which is the desired equivalence.
\end{proof}

\begin{corollary}\label{x0daz5}
  An object of \(\CAlg(\Cat{Pr}^{\con})\)
  is a presentably symmetric monoidal \(\infty\)-category
  that is continuous
  such that \(\unit\) is compact
  and \(C\otimes D\to C'\otimes D'\) is compact
  whenever \(C\to C'\) and \(D\to D'\) are compact morphisms.
  A morphism between objects \(\cat{C}\) and \(\cat{D}\in\CAlg(\Cat{Pr}^{\con})\)
  is a colimit-preserving symmetric monoidal functor
  \(\cat{C}\to\cat{D}\)
  that preserves compact morphisms.
\end{corollary}

\subsection{Stable continuous categories}\label{ss:st}

We recall the following notion:

\begin{definition}\label{xxsc2l}
  We write \(\Cat{Dbl}\)
  for the subcategory
  of~\(\Cat{Pr}_{\st}\)
  spanned by dualizable objects (see \cref{xz1vid})
  and morphisms admitting right adjoints in~\(\Cat{Pr}_{\st}\).
\end{definition}

Lurie proved the following in~\cite[Section~D.7]{LurieSAG},
which connects this notion to our previous discussions:

\begin{theorem}[Lurie]\label{x2hild}
  The subcategory \(\Cat{Dbl}\subset\Cat{Pr}\)
  is the intersection of~\(\Cat{Pr}^{\con}\)
  and~\(\Cat{Pr}_{\st}\).
\end{theorem}

\begin{remark}\label{xvg681}
  Note that compact morphisms are
  \emph{not} closed under extensions or (co)fibers.
  For example,
  for any object~\(C\),
  the morphism
  \({\id}\colon C\to C\)
  is an extension
  of \(C\to 0\)
  by \(0\to C\).
\end{remark}

Still,
in some cases,
compactness interacts favorably
with the formation
of (co)fibers:

\begin{proposition}\label{cpt_st}
  Let \(C\to C'\to D\) be morphisms
  in a stable continuous presentable \(\infty\)-category.
  Consider the morphism
  \begin{equation*}
    E=\cofib(C\to D)
    \to
    \cofib(C'\to D)=E'.
  \end{equation*}
  When \(D\) is compact,
  \(C\to C'\) is compact
  if and only if \(E\to E'\) is compact.
\end{proposition}

\begin{lemma}\label{x1rr9b}
  Let \(\cat{C}\) be a stable presentable \(\infty\)-category.
  We write \(\map\) for the mapping spectrum.
  We call a morphism \(\map\)-compact if the condition in \cref{xl2h0i}
  is satisfied when we replace~\(\Map\) with~\(\map\).
  \begin{enumerate}
    \item\label{i:obv_st}
      If a morphism is \(\map\)-compact,
      then it is compact.
    \item\label{i:rev_st}
      If \(\cat{C}\) is continuous,
      the converse holds.
  \end{enumerate}
\end{lemma}

\begin{proof}
  First, \cref{i:obv_st} follows from just taking~\(\Omega^{\infty}\).
  We prove~\cref{i:rev_st}.
  Suppose that \(\cat{C}\) is \(\kappa\)-compactly generated
  and \(C\to C'\) is a compact morphism
  between \(\kappa\)-compact objects.
  We consider the Yoneda embedding \(j\colon\cat{C}\to\Ind(\cat{C}_{\kappa})\),
  which has a twice left adjoint by \cref{x9m28r}.
  By \cref{x11iiz},
  the map \(\check{\jmath}(C)\to\check{\jmath}(C')\)
  factors through~\(j(C'')\) for some \(C''\in\cat{C}_{\kappa}\).
  Then our desired result follows
  by considering the \(\map\)-version of \cref{xq60rx,x11iiz}.
\end{proof}

\begin{proof}[Proof of \cref{cpt_st}]
  The statement immediately follows from \cref{x1rr9b}.
\end{proof}

\section{Stably compact spaces and sheaves thereon}\label{s:sc}

Specializing the discussion in \cref{s:con} to the posetal setting
yields the theory of continuous posets and stably compact locales.
We recall the definition of stably compact spaces
in \cref{ss:sc_def},
study basic operations
in \cref{ss:sc_op},
and then proceed to the study of sheaves
in \cref{ss:sheaf}.
We prove proper base change
and Verdier duality for these spaces
in \cref{ss:ver}.

\begin{remark}\label{xll2rj}
  Our treatment here is ahistorical;
  the notion of continuous poset—which itself was derived from
  the notion of continuous lattice
  of Scott~\cite{Scott70}—was introduced earlier,
  and then
  the notion of continuous categories
  was introduced by
  Johnstone–Joyal~\cite{JohnstoneJoyal82}
  as its categorification.
\end{remark}

\subsection{Stably compact spaces}\label{ss:sc_def}

Recall that a \emph{frame} is an algebraic structure
that axiomatizes the lattice of open sets of a topological space.
It is a complete poset
satisfying the distributivity law
\(
\bigvee_{i\in I}(x\wedge y_i)
=
x\wedge\bigvee_{i\in I}y_i
\) for any \(x\) and \((y_i)_{i\in I}\).
We can regard the category \(\Cat{Frm}\)
as a full subcategory
of \(\CAlg(\Cat{Pr})\):
Namely, a presentably symmetric monoidal \(\infty\)-category
is a frame
if and only if
the \(\infty\)-category
is (essentially) a poset\footnote{Note that accessible \(0\)-categories are small.
} and
the symmetric monoidal structure is cartesian.
We make the following definition:

\begin{definition}\label{xxw5fc}
  We define the category of \emph{stably continuous frames}
  to be the pullback
  \begin{equation*}
    \begin{tikzcd}
      \Cat{SCFrm}\ar[r,hook]\ar[d]&
      \CAlg(\Cat{Pr}^{\con})\ar[d]\\
      \Cat{Frm}\ar[r,hook]&
      \CAlg(\Cat{Pr})\rlap.
    \end{tikzcd}
  \end{equation*}
  We write \(\Cat{SC}\)
  for the opposite of \(\Cat{SCFrm}\)
  and call it the category of \emph{stably compact locales}.
\end{definition}

\begin{example}\label{xmavip}
  When we replace \(\Cat{Pr}^{\con}\)
  with \(\Cat{Pr}^{\cg}\)
  in \cref{xxw5fc},
  we get the category of coherent frames.
  Its opposite category
  is the category of spectral spaces (and quasicompact maps) \(\Cat{Spec}\).
  Therefore,
  \(\Cat{Spec}\)
  is a full subcategory of \(\Cat{SC}\).
\end{example}

We have the following concrete characterization
(see \cref{way_below} for the symbol \(\ll\)):

\begin{proposition}\label{xbmydj}
  A frame~\(R\) is stably continuous
  if and only if the following conditions are satisfied:
  \begin{conenum}
    \item\label{i:con_con}
      For any~\(x\), we have
      \(x=\bigvee_{y\ll x}y\).
    \item\label{i:con_alg}
      For any~\(x\) and~\(y_{1}\), …,~\(y_{n}\),
      we have
      \(x\ll y_{1}\wedge\dotsb\wedge y_{n}\)
      whenever \(x\ll y_{i}\) for all \(i\).
  \end{conenum}
  For a frame morphism \(f\colon R\to S\)
  between stably continuous frames,
  it is a morphism
  of stably continuous frame
  if and only if
  \(x\ll y\) implies \(f(x)\ll f(y)\).
\end{proposition}

\begin{proof}
  We prove the first part.
  By \cref{con_0},
  for a complete poset,
  it is continuous if and only if
  \(x=\bigvee_{y\ll x}y\) holds for any~\(x\),
  which explains \cref{i:con_con}.
  Then \cref{x0daz5}
  explains \cref{i:con_alg}.

  The second part follows from \cref{x0daz5}.
\end{proof}

Note that
the theory becomes much simpler,
since
\(\otimes\) coincides with finite products
in this situation:

\begin{proposition}\label{xmx7uh}
  A frame~\(R\) is stably continuous
  if and only if it is a retract (as a frame) of a spectral frame.
\end{proposition}

\begin{proof}
  The “only if” direction is immediate.
  We prove the “if” direction.
  We consider the canonical map \(\Ind(R)\to R\) of frames.
  It admits a left adjoint~\(\check{\jmath}\) by \cref{x9m28r}.
  By identifying \(\Ind(R)\) with the frame of ideals\footnote{Recall that for a poset~\(P\),
    its \(\Ind\)-objects correspond to its \emph{ideals};
    when \(P\) has finite joins,
    an ideal is a downward subset closed under finite joins.
  },
  we see that \(\check{\jmath}(y)\) is \(\{x\in R\mid x\ll y\}\).
  We wish to prove that \(\check{\jmath}\) is a morphism of frames,
  i.e., that \(\check{\jmath}(y_{1}\wedge\dotsb\wedge y_{n})
  \leq\check{\jmath}(y_{1})\cap\dotsb\cap\check{\jmath}(y_{n})\)
  is equality.
  This is exactly~\cref{i:con_alg} of~\cref{xbmydj}.
\end{proof}

The same argument shows the following:

\begin{proposition}\label{xdzqu2}
  Let \(R\to S\) be a frame morphism
  between stably continuous frames.
  The following are equivalent:
  \begin{conenum}
    \item\label{i:okay}
      It is in \(\Cat{SCFrm}\).
    \item\label{i:can2}
      The square
      \begin{equation*}
        \begin{tikzcd}
          \Ind(R)\ar[r]\ar[d]&
          R\ar[d]\\
          \Ind(S)\ar[r]&
          S
        \end{tikzcd}
      \end{equation*}
      is left adjointable.
    \item\label{i:ret2}
      It is a retract (as a frame morphism)
      of a morphism in \(\Cat{Spec}\).
  \end{conenum}
\end{proposition}

\begin{proof}
  This follows from \cref{mor_adj,x0daz5}.
\end{proof}

Since \(\Cat{Spec}\)
can be realized as a (nonfull) subcategory of topological spaces,
we deduce the following spatiality result:

\begin{corollary}\label{aa91d6}
  We can realize \(\Cat{SC}\)
  as a (nonfull) subcategory of the category of topological spaces.
  With this realization,
  we call it the category of \emph{stably compact spaces}.
\end{corollary}

\begin{remark}\label{concrete}
  Regarding \cref{aa91d6},
  we have the following more precise description:
  \begin{itemize}
    \item
      A topological space is stably compact
      if
      it is sober, compact, locally compact,
      and the intersection of two compact saturated\footnote{A subset of a topological space is called \emph{saturated}
        if it is closed under generalizations;
        e.g., open subsets are saturated.
      } subsets is compact.
    \item
      A continuous map between them is a morphism,
      which we call a \emph{perfect} map,
      if the inverse image of any compact saturated set
      is compact.
\end{itemize}
Note that some literature calls the morphisms “proper” instead,
  but this conflicts the notion of properness
  in locale theory (cf.~\cref{ss:ver}).
  For example,
  the inclusion of the open point
  to the Sierpiński space is perfect,
  but not closed.
\end{remark}

\begin{example}\label{xccv9o}
  Consider the \emph{directed interval}~\(\II\),
  whose underlying set is \([0,1]\)
  with the topology whose nontrivial open sets
  are of the form \((r,1]\) for \(r \in [0,1)\).
  This is a stably compact space.
\end{example}

\subsection{Operations on stably compact spaces}\label{ss:sc_op}

We review basic operations on stably compact spaces.
We begin by the following,
which already appeared in the proof of \cref{xmx7uh}:

\begin{definition}\label{xs5a86}
  Let \(X\) be a stably compact locale
  whose frame is \(R\).
  We write \(\gamma X\)
  for the locale corresponding to
  the frame \(\Ind(R)\)
  and call it
  the \emph{Flachsmeyer construction}.
  As \cref{xmx7uh} shows,
  there are canonical locale morphisms in both directions,
  whereas the morphism
  \(\gamma X\to X\) is a morphism of stably compact locales.
\end{definition}

\begin{remark}\label{x6spu8}
  To my knowledge,
  \cref{xs5a86} was first considered by Flachsmeyer~\cite{Flachsmeyer61}
  for topological spaces,
  hence the name we use here.
\end{remark}

\begin{remark}\label{x5dfgr}
  In \cref{xs5a86},
  this situation
  of~\(p_{*}\) having a fully faithful right adjoint
  is called a \emph{local} geometric morphism
  in topos theory.
\end{remark}

With this,
we rewrite \cref{xdzqu2} as follows:

\begin{corollary}\label{xq6bfr}
  Let \(X\) and \(Y\) be stably compact locales
  and \(f\colon Y\to X\) a map of locales.
  Then \(f\) is a morphism of stably compact locales
  if and only if
  the canonical \(2\)-cell in
  \begin{equation*}
    \begin{tikzcd}
      \gamma Y\ar[r]\ar[d]&
      Y\ar[d]\\
      \gamma X\ar[r]&
      X
    \end{tikzcd}
  \end{equation*}
  is invertible.
\end{corollary}

\begin{definition}\label{xv06ia}
  For a stably compact space~\(X\),
  we define~\(X^{\#}\)
  to be the topological space
  whose open sets are generated
  by the open sets of~\(X\)
  and the complements of compact saturated sets of~\(X\).
  We call this topology the \emph{patch topology}.
\end{definition}

\begin{example}\label{xxqapw}
  The patch topology on~\(\II\) (see \cref{xccv9o})
  coincides with the standard topology on~\([0,1]\).
\end{example}

\begin{definition}\label{x7pj79}
  A \emph{Nachbin space} is a compact Hausdorff space~\(X\)
  with a partial order on the underlying set such that
  \({\leq}=\{(x,x')\mid x\leq x'\}\subset X\times X\)
  is closed.
  A morphism is an order-preserving continuous map.
  We write \(\Cat{Nach}\) for the category of Nachbin spaces.
\end{definition}

\begin{example}\label{xpq4ys}
  The usual ordering on the compact Hausdorff space~\([0,1]\)
  determines a Nachbin space.
  In fact,
  this is what corresponds to~\(\II\) (see \cref{xccv9o})
  under the equivalence of \cref{xf8mf0} below.
\end{example}

We use the following,
which first appeared
in~\cite[Exercises~VII.1.18–19]{GHKLMS}:

\begin{theorem}[Gierz–Lawson]\label{xf8mf0}
For a stably compact space~\(X\),
  the pair of \(X^{\#}\)
  and the specialization order
  constitute a Nachbin space.
  Moreover,
  this assignment gives an equivalence of categories
  \(\Cat{SC}\to\Cat{Nach}\).
\end{theorem}

\begin{remark}\label{xglrlq}
The equivalence of \cref{xf8mf0}
  can be upgraded to an equivalence
  between \((1,2)\)-categories.
\end{remark}

\begin{definition}\label{xoc8bw}
  On \(\Cat{Nach}\)
  we have an autoequivalence
  by swapping the order.
  The \emph{de~Groot dual} functor~\((\X)^{\op}\)
  is the corresponding autoequivalence on~\(\Cat{SC}\)
  induced by \cref{xf8mf0}.
\end{definition}

\begin{remark}\label{xx808s}
  In terms of \cref{concrete},
  from a stably compact space~\(X\),
  an open subset
  of its de~Groot dual
  is the complement of
  a compact saturated set.
\end{remark}

\begin{example}\label{xsc2iw}
  The nontrivial open sets of
  the de~Groot dual of~\(\II\) (see \cref{xccv9o})
  are of the form \([0,r)\) for \(r\in(0,1]\).
\end{example}

\begin{remark}\label{x6duhh}
  Note that \cref{xoc8bw} describes the only nontrivial automorphism
  of~\(\Cat{SC}\) (as an \(\infty\)-category).
  More precisely,
  the space \(\Aut(\Cat{SC})\) is equivalent to~\(C_{2}\),
  as we can see as follows:
  Since any autoequivalence must preserve the final object,
  the underlying set of each space remains unchanged.
  Furthermore, by examining the two-point cases closely,
  we see that the Sierpiński space is mapped to itself.
  From these facts,
  we see that any autoequivalence restricts to \(\Cat{Spec}\).
  However, again,
  any automorphism of \(\Cat{Spec}\) is fixed by what it does on generators,
  so we see that \(\pi_{0}(\Aut(\Cat{Spec}))\) is \(C_{2}\).
Since any stably compact space is the colimit
  of some diagram of spectral spaces,
  an autoequivalence of \(\Cat{Spec}\)
  extends (at most) uniquely to an autoequivalence of \(\Cat{SC}\).
  Consequently, \(\pi_{0}(\Aut(\Cat{SC}))\) is~\(C_{2}\).
  Also, it is easy to see that \(\Aut(\id_{\Cat{SC}})\) is trivial,
  again by looking at the Sierpiński space.
\end{remark}

\begin{remark}\label{xw0gri}
  If we consider \(\Cat{SC}\) as a \((1,2)\)-category
  using the \((1,2)\)-categorical structure on~\(\Cat{Loc}\),
  the de~Groot duality functor
  is an equivalence \(\Cat{SC}^{\co}\to\Cat{SC}\).
\end{remark}

\begin{remark}\label{xwhlbz}
  In \cref{ss:ver},
  we give a sheaf-theoretic explanation
  of the duality in \cref{xoc8bw}.
\end{remark}

\begin{example}\label{x95c5f}
  Let \(X\) be a spectral space.
  We have the following:
  \begin{itemize}
    \item
      Its de~Groot dual is again spectral,
      and called the \emph{Hochster dual}.
    \item
      The patch topology
      is commonly called the \emph{constructible topology}.
    \item
      Its Flachsmeyer construction
      is the conjugate of the \(\Pro\)-Zariski topology;
      i.e., \((\gamma(X^{\op}))^{\op}\)
      is (the corresponding locale of)
      the \(\Pro\)-Zariski topos of~\(X\).
    \item
      The Stone–Čech compactification
      is \emph{not} \(\pi_{0}X\);
      see \cref{xmd5ll}.
  \end{itemize}
\end{example}

\begin{example}\label{xumbdo}
  Let \(X\) be a compact Hausdorff space.
  Then the de~Groot dual, patch topology,
  Stone–Čech construction
  leave the space unchanged.
  The Flachsmeyer construction
  gives a spectral space
  whose quasicompact opens
  are opens in~\(X\).
\end{example}

\begin{remark}\label{xqjsp7}
  Gabriel–Ulmer~\cite{GabrielUlmer71}
  proved that \(\Cat{CH}\) is \(\aleph_{1}\)-cocompactly generated.
Similarly,
  one can show that \(\Cat{SC}\)
  is \(\aleph_{1}\)-cocompactly generated.
  Moreover,
  as in~\cite[Proposition~2.11]{k-ros-1},
  one can obtain
  a characterization of \(\kappa\)-cocompacts
  for any regular cardinal~\(\kappa\).
  Cocompact objects are precisely
  finite spectral spaces.
  When \(\kappa\) is uncountable,
  \(\kappa\)-cocompact objects are precisely
  those whose weights are less than~\(\kappa\).
\end{remark}

\subsection{Sheaves on stably compact spaces}\label{ss:sheaf}

\begin{theorem}\label{xc8o0g}
  The functor
  \({\Shv}\colon\Cat{Frm}\to\CAlg(\Cat{Pr})\)
  restricts to
  \(\Cat{SCFrm}\to\CAlg(\Cat{Pr}^{\con})\).
\end{theorem}

\begin{proof}
  It suffices to check that
  \({\Shv}\colon\Cat{Loc}^{\op}\to\Cat{Pr}\)
  restricts to
  \({\Shv}\colon\Cat{SC}^{\op}\to\Cat{Pr}^{\con}\).
  On the level of objects,
  \(\Shv(X)\) is a retract of \(\Shv(\gamma X)\)
  in~\(\Cat{Pr}\),
  and therefore we see that it is continuous
  using \cref{xl9gxj} below.
  On the level of morphisms,
  this follows from \cref{xq6bfr}
  and using \cref{xl9gxj} again.
\end{proof}

This proof is not direct
and does not reveal much about the internal structure
of \(\Shv(X)\).
To understand \(\Shv(X)\) more concretely,
the Flachsmeyer construction \(p\colon\gamma X\to X\)
and its conjugate
\(q\colon(\gamma(X^{\op}))^{\op}\to X\)
are useful.
We can describe sheaves on~\(X\)
in different ways:
\begin{itemize}
  \item
    In \cref{shv_fl},
    we realize it as
    the essential image of~\(p^{*}\)
    inside \(\Shv(\gamma X)\).
  \item
    In \cref{shv_cofl},
    we realize it as
    the essential image of~\(q^{*}\)
    inside \(\Shv((\gamma(X^{\op}))^{\op})\);
    when \(X\) is Hausdorff,
    this description is called “K-sheaves”;
    see \cref{xjlmvy}.
  \item
    We can also realize it as
    the essential image of~\(p^{!}\),
    which is the right adjoint of~\(p_{*}\),
    but this is not special to our situation;
    see \cref{xrmpdj}.
\end{itemize}
We do this for more general coefficients
for the sake of completeness.
First,
we recall how to describe
sheaves on spectral spaces
from~\cite[Theorem~7.3.5.2]{LurieHTT}:

\begin{proposition}[Lurie]\label{xjhzdx}
  Let \(\cat{C}\) be a presentable \(\infty\)-category
  and \(X\) a spectral space.
  Let \(D\) be the distributive lattice of compact open sets of~\(X\).
  Then the restriction
  \(\Shv(X;\cat{C})\to\Fun(D^{\op},\cat{C})\)
  is fully faithful and its essential image
  consists of those functors~\(\shf{F}\colon D^{\op}\to\cat{C}\)
  satisfying the following conditions:
  \begin{conenum}
    \item\label{i:reduced}
      It maps \(\emptyset\) to the final object.
    \item\label{i:excisive}
      For any elements~\(U\) and~\(V\),
      the square
      \begin{equation*}
        \begin{tikzcd}
          \shf{F}(U\cup V)\ar[r]\ar[d]&
          \shf{F}(U)\ar[d]\\
          \shf{F}(V)\ar[r]&
          \shf{F}(U\cap V)
        \end{tikzcd}
      \end{equation*}
      is cartesian.
   \end{conenum}
\end{proposition}

This immediately implies the following:

\begin{corollary}\label{xl9gxj}
  The functor \({\Shv}\colon\Cat{Frm}\to\CAlg(\Cat{Pr})\)
  restricts to \(\Cat{Spec}\to\CAlg(\Cat{Pr}^{\cg})\).
\end{corollary}

\begin{definition}\label{x9h7sm}
  Let \(X\) be a locale
  whose frame is~\(R\).
  We call a presheaf \(R^{\op}\to\cat{C}\)
  valued in an \(\infty\)-category with finite limits
  \emph{excisive}\footnote{In some contexts,
    the term “excisive” refers only to~\cref{i:excisive},
    while the term “reduced” is used for~\cref{i:reduced}.
  }
  when it satisfies \cref{i:reduced,i:excisive} of \cref{xjhzdx}.
\end{definition}

\begin{example}\label{xrmpdj}
  Consider a locale~\(X\)
  and a presentable \(\infty\)-category~\(\cat{C}\).
  We have seen in~\cite[Theorem~3.19]{ttg-sm}
  that a \(\cat{C}\)-valued presheaf on~\(X\) is a sheaf
  if and only if it is excisive
  and preserves cofiltered limits.
  In fact,
  this corresponds to
  describing \(\Shv(X;\cat{C})\)
  as a full subcategory
  using the pushforward along \(X\to\gamma X\)—note that
  the definition of \(\gamma X\)
  makes sense for any locale,
  but we only obtain a morphism \(X\to\gamma X\).
  When \(X\) is stably compact,
  this correspond to describing
  the image of~\(p^{!}\),
  which is right adjoint to~\(p_{*}\),
  for \(p\colon\gamma X\to X\).
\end{example}

\begin{proposition}\label{shv_fl}
  Consider a stably compact space~\(X\)
  and a presentable \(\infty\)-category~\(\cat{C}\).
  Suppose that finite limits
  commute with filtered colimits in~\(\cat{C}\).
  Let \(\shf{G}\) be a \(\cat{C}\)-valued sheaf on~\(\gamma X\).
  It is in the image of \(p^{*}\colon\Shv(X)\to\Shv(\gamma X)\)
  if and only if
  \begin{equation}
    \label{e:i9uc6}
    \injlim_{V\gg U}\shf{G}(V)
    \to
    \shf{G}(U)
  \end{equation}
  is an equivalence for any open~\(U\) of~\(X\).
\end{proposition}

\begin{proof}
Recall that a \(\cat{C}\)-valued sheaf on \(\gamma X\)
  is an excisive functor \(R^{\op} \to \cat{C}\),
  where \(R\) denotes the frame of open subsets of~\(X\),
  by \cref{xjhzdx}.
  By adjunction, \(\shf{G}\) lies in the essential image of \(p^{*}\)
  if and only if the counit map
  \(p^{*}p_{*}\shf{G} \to \shf{G}\)
  is an equivalence.
  We compute this.

  The functor \(p_{*}\shf{G}\)
  sends an open~\(V\) of~\(X\)
  to \(\projlim_{W\ll V}\shf{G}(W)\),
  and thus \(p^{*}p_{*}\shf{G}\) is the sheafification
  of the presheaf
  \begin{equation*}
    U\mapsto\injlim_{V\gg U}\projlim_{W\ll V}\shf{G}(W),
  \end{equation*}
  where \(U\) is an open of~\(X\).
  We prove that this is equivalent to
  the left-hand side of \cref{e:i9uc6},
  which is already excisive
  by our assumption on~\(\cat{C}\).

  Therefore,
  it suffices to show that
  the morphism
  \begin{align*}
    \projlim_{V\gg U}\injlim_{W\ll V}W
    \to
    \projlim_{V\gg U}V
  \end{align*}
  is an equivalence in \(\Pro(\Ind(R))\).
  Since it is a poset,
  we need to obtain a morphism in the other direction,
  but we can do this
  by observing that
  for any \(U\ll V\),
  there is~\(W\)
  satisfying \(U\ll W\ll V\).
\end{proof}

\begin{proposition}\label{shv_cofl}
Consider a stably compact space~\(X\)
  and a presentable \(\infty\)-category~\(\cat{C}\).
  Suppose that finite limits
  commute with filtered colimits in~\(\cat{C}\).
  Let \(\shf{G}\) be a \(\cat{C}\)-valued sheaf on~\(\gamma X\).
  A \(\cat{C}\)-valued sheaf on \((\gamma(X^{\op}))^{\op}\)
  is in the essential image of~\(q^{*}\)
  if and only if
  \begin{equation}
    \label{e:iv8wp}
    \injlim_{L\gg'K}\shf{G}(L)\to\shf{G}(K)
  \end{equation}
  is an equivalence
  for a compact saturated set~\(K\) of~\(X\).
  Here, \(K\ll'L\) for compact saturated sets
  means that \(X\setminus L\ll X\setminus K\) holds
  in \(X^{\op}\);
  i.e., that there is an open in~\(X\) such that
  \(K\subset U\subset L\).
\end{proposition}

\begin{remark}\label{xjlmvy}
  When \(X\) is Hausdorff,
  \cref{shv_cofl} recovers
  Lurie’s notion of “K-sheaves” in~\cite[Section~7.3.4]{LurieHA}.
\end{remark}

\begin{proof}
Recall that a \(\cat{C}\)-valued sheaf on \((\gamma(X^{\op}))^{\op}\)
  is an excisive functor \(Q^{\op} \to \cat{C}\),
  where \(Q\) denotes the coframe of compact saturated subsets of~\(X\),
  by \cref{xjhzdx}.
  By adjunction, \(\shf{G}\) lies in the essential image of \(q^{*}\)
  if and only if the counit map
  \(q^{*}q_{*}\shf{G} \to \shf{G}\) is an equivalence.
  We compute this.

  The functor \(q_{*}\shf{G}\) sends an open~\(U\) of~\(X\)
  to \(\projlim_{L\subset U} \shf{G}(L)\),
  where \(L\) runs over compact saturated subsets.
  Thus \(q^{*}q_{*}\shf{G}\) is the sheafification
  of the presheaf
  \begin{equation*}
    K\mapsto
    \injlim_{U\supset K} \projlim_{L \subset U} \shf{G}(L)
  \end{equation*}
  where \(K\) is a compact saturated set of~\(X\).
  We prove that this is equivalent to
  the left-hand side of \cref{e:iv8wp},
  which is already excisive
  by our assumption on~\(\cat{C}\).

The proof is similar to that of \cref{shv_fl},
  but this time,
  we do not have obvious morphisms
  between
  \(\projlim_{U\supset K}\injlim_{L\subset U}L\)
  and
  \(\projlim_{L\gg' K}L\).
  Nevertheless,
  again, since \(\Pro(\Ind(Q))\) is a poset,
  it suffices to explicitly construct them.
  The morphism from the first to the second
  is obtained from the fact that
  \(K\ll'L\) implies
  that there is an open~\(U\) such that \(K\subset U\subset L\).
  The morphism in the other direction
  is obtained from the fact that
  when \(K\subset U\),
  there is a compact saturated set~\(L\)
  such that \(K\ll'L\subset U\).
\end{proof}

We note the following,
which is useful:

\begin{proposition}\label{sheafify}
  Let \(\cat{C}\) be a presentable \(\infty\)-category
  and \(\shf{F}\) a \(\cat{C}\)-valued
  presheaf on a stably compact space~\(X\).
  Suppose that \(\shf{F}\) is excisive.
  Then the morphism
  \begin{equation*}
    \shf{F}(U)
    \to
    \projlim_{V\ll U}\shf{F}(V)
  \end{equation*}
  for an open~\(U\) of~\(X\)
  is the sheafification morphism
  evaluated at~\(U\).
\end{proposition}

\begin{proof}
  We write \((\X)'\) for the sheafification.
  We aim to compute \(\shf{F}'\).
  This is equivalent to
  \(p_{*}p^{*}(\shf{F}')\),
  where \(p\colon\gamma X\to X\)
  denotes the Flachsmeyer construction.
  This can also be computed as
  \(p_{*}((p^{*}\shf{F})')\),
  where \(p^{*}\) denotes
  the presheaf pullback.
  Since \(\shf{F}\) is excisive,
  for any open~\(U\) of~\(X\),
  regarded as a quasicompact open of \(\gamma X\),
  we have \((p^{*}\shf{F})'(U)\simeq\shf{F}(U)\).
  Hence we obtain the desired result.
\end{proof}

\begin{remark}\label{x65mfc}
  One can also prove \cref{sheafify} directly
  without relying on the Flachsmeyer construction,
  though the argument is more involved.
\end{remark}

\begin{corollary}\label{xcy9f7}
  Let \(X\) be a stably compact space.
  Then for opens \(U\ll U'\),
  the morphism \(U\to U'\)
  in \(\Shv(X)\) is compact.
\end{corollary}

\begin{proof}
  Note that
  filtered colimits of excisive presheaves
  stay excisive,
  since
  filtered colimits commute with finite limits
  in~\(\Cat{S}\).
  Thus we may apply \cref{sheafify}
  to compute the filtered colimit of sheaves
  \(\injlim_{i}\shf{F}_{i}\).
  We need a filler in the diagram
  \begin{equation*}
    \begin{tikzcd}
      \injlim_{i}(\shf{F}_{i}(U'))\ar[r]\ar[d]&
      \injlim_{i}(\shf{F}_{i}(U))\ar[d]\\
      \projlim_{V'\ll U'}\injlim_{i}(\shf{F}_{i}(V'))\ar[r]\ar[ur,dashed]&
      \projlim_{V\ll U}\injlim_{i}(\shf{F}_{i}(V))\rlap.
    \end{tikzcd}
  \end{equation*}
  From the inclusions of ideals
  \begin{equation*}
    \{V\mid V\ll U\}
    \subset\{V\mid V\leq U\}
    \subset\{V'\mid V'\ll U'\}
    \subset\{V'\mid V'\leq U'\},
  \end{equation*}
  we get a canonical filler.
\end{proof}

\subsection{Digression: proper base change and Verdier duality}\label{ss:ver}

Lurie proved the following important properties
about sheaves on compact Hausdorff spaces
in~\cite[Section~7.3.4]{LurieHTT} and~\cite[Section~5.5.5]{LurieHA}:

\begin{theorem}[Lurie]\label{xuizxc}
  Let \(X\) be a compact Hausdorff space.
  \begin{enumerate}
    \item\label{i:pbc}
      The \(\infty\)-topos \(\Shv(X)\) is proper
      (see \cref{xwri4g}).
    \item\label{i:ver}
      For a stable \(\infty\)-category~\(\cat{C}\)
      having limits and colimits,
      there is a canonical equivalence
      \({\VD}\colon\Shv(X;\cat{C})\to\cShv(X;\cat{C})\)
      pointwise given by
      \begin{equation*}
        \VD(\shf{F})(U)
        =
        \injlim_{V\supset X\setminus U}\fib
        \bigl(
          \shf{F}(X)\to\shf{F}(V)
        \bigr),
      \end{equation*}
      where \(\cShv(X;\cat{C})=\Shv(X;\cat{C}^{\op})^{\op}\)
      denotes the \(\infty\)-category
      of cosheaves.
  \end{enumerate}
\end{theorem}

\begin{remark}\label{x10lcy}
Note that
  \cref{i:ver} of \cref{xuizxc}
  was proven for locally compact Hausdorff spaces,
  but that version can be deduced from the compact case.
\end{remark}

In this section,
we generalize these statements
to stably compact spaces.

\begin{theorem}\label{sc_proper}
  For a stably compact space~\(X\),
  the \(\infty\)-topos \(\Shv(X)\) is proper.
\end{theorem}

\begin{remark}\label{xp9f04}
  Lurie proved this when \(X\) is spectral
  in~\cite[Corollary~7.3.5.3]{LurieHTT}.
  So \cref{sc_proper} also generalizes this result.
\end{remark}

\begin{theorem}\label{sc_verdier}
  Let \(X\) be a stably compact space.
  Then there is a canonical equivalence
  \({\VD}\colon\Shv(X;\Cat{Sp})\to\cShv(X^{\op};\Cat{Sp})\)
  pointwise given by
  \begin{equation*}
    \VD(\shf{F})(X\setminus K)
    =\injlim_{U\supset K}\fib(\shf{F}(X)\to\shf{F}(U))
  \end{equation*}
  where \(K\) is a compact saturated subset of~\(X\).
\end{theorem}

\begin{remark}\label{xmdzsq}
  In \cref{sc_verdier},
  we can replace~\(\Cat{Sp}\) with
  any stable presentable \(\infty\)-category~\(\cat{C}\).
  This can be done by tensoring~\(\cat{C}\),
  since we know that \(\Shv(X;\Cat{Sp})\)
  is dualizable by \cref{xc8o0g,x2hild}.
\end{remark}

We first recall the notion of properness
from~\cite[Section~7.3.1]{LurieHTT}:

\begin{definition}\label{xwri4g}
  Let 
  \(p\colon\cat{Y}\to\cat{X}\)
  be a geometric morphism
  between \(\infty\)-toposes.
  We say that
  \(p\) is \emph{proper}
  if for any cartesian squares
  of \(\infty\)-toposes
  \begin{equation*}
    \begin{tikzcd}
      \cat{Y}''\ar[r,"g'"]\ar[d,"p''"']&
      \cat{Y}'\ar[r,"g"]\ar[d,"p'"]&
      \cat{Y}\ar[d,"p"]\\
      \cat{X}''\ar[r,"f'"]&
      \cat{X}'\ar[r,"f"]&
      \cat{X}\rlap,
    \end{tikzcd}
  \end{equation*}
  the canonical natural transformation
  \((f')^{*}\circ p'_{*}\to p''_{*}\circ(g')^{*}\) is an equivalence.
  When \(\cat{X}\) is~\(\Cat{S}\),
  we simply say that \(\cat{Y}\) is \emph{proper}.
\end{definition}

\begin{remark}\label{xhdrl5}
  A geometric morphism between \(1\)-toposes
  satisfying the condition of \cref{xwri4g}
  when “\(\infty\)-topos” is replaced by “\(1\)-topos”
  is usually called \emph{tidy} in \(1\)-topos theory literature.
\end{remark}

\begin{remark}\label{0_proper}
  The sheaf \(\infty\)-topos of
  a proper locale,
  i.e., a locale satisfying the condition of \cref{xwri4g}
  when “\(\infty\)-topos” is replaced by “locale”,
  is \emph{not} proper in general.
  This makes~\cref{i:pbc} of \cref{xuizxc}
  subtle.
\end{remark}

\begin{proof}[Proof of \cref{sc_proper}]
  Let \(\cat{Y}\to\cat{X}\) be a geometric morphism.
  We obtain the diagram
  \begin{equation*}
    \begin{tikzcd}
      \Shv(X;\cat{Y})\ar[r,"g"]\ar[d,"q"']&
      \Shv(X;\cat{X})\ar[d,"p"]\\
      \cat{Y}\ar[r,"f"]&
      \cat{X}\rlap.
    \end{tikzcd}
  \end{equation*}
  We need to show that
  \(f^{*}p_{*}\shf{F}\to q_{*}g^{*}\shf{F}\)
  is an equivalence
  for any object~\(\shf{F}\) of~\(\Shv(X;\cat{X})\).

  We write \(R\) for the frame of opens of~\(X\).
  We use \cref{shv_fl}
  to identify
  \(\Shv(X;\cat{X})\)
  with the full subcategory of
  excisive functors \(R^{\op}\to\cat{X}\)
  such that \cref{e:i9uc6} is an equivalence.
  We similarly proceed for \(\Shv(X;\cat{Y})\).
  Then \(g^{*}\)
  can be computed as objectwise
  applying~\(f^{*}\),
  since \(f^{*}\)
  commutes with finite limits and filtered colimits.
  Therefore,
  we obtain the desired result.
\end{proof}

Our proof above is direct,
but
more conceptual proofs are also possible:

\begin{remark}\label{xogrs7}
  To prove \cref{sc_proper},
  by \cref{xmx7uh},
  we can assume that \(X\) is spectral.
  This case was already established by Lurie;
  see \cref{xp9f04}.
\end{remark}

\begin{remark}\label{xco75c}
  Martini–Wolf~\cite{MartiniWolf25}
  recently proved the characterization of proper geometric morphisms
  expected in~\cite[Remark~7.3.1.5]{LurieHTT}
  (and known in \(1\)-topos theory).
  In particular,
  an \(\infty\)-topos~\(\cat{X}\) is proper if and only if
  the final object is compact.
  Hence, \cref{sc_proper}
  follows from \cref{xc8o0g}.
  However note that
  given how the proof in~\cite{MartiniWolf25}
  goes,
  this proof is similar to \cref{xogrs7};
  cf.~\cite[Remark~3.5.5]{MartiniWolf25}.
  Nevertheless,
  the reader is encouraged to consult their paper
  for a more satisfactory treatment about properness
  in general.
\end{remark}

We now move on to proving Verdier duality:

\begin{proof}[Proof of \cref{sc_verdier}]
  We write~\(Q\)
  for the coframe of compact saturated subsets of~\(X\).
  We use \cref{shv_cofl}
  to identify \(\Shv(X;\Cat{Sp})\)
  with the \(\infty\)-category~\(\cat{C}\)
  of excisive functors \(\shf{G}\colon Q^{\op}\to\Cat{Sp}\)
  such that \cref{e:iv8wp} is an equivalence.
  We use \cref{shv_fl}
  to identify \(\cShv(X;\Cat{Sp})\)
  with the \(\infty\)-category~\(\cat{C}'\)
  of excisive functors \(\shf{G}'\colon Q\to\Cat{Sp}\)
  such that \cref{e:i9uc6} is an equivalence.
  The functor \(\cat{C}\to\cat{C}'\)
  given by
  \begin{equation*}
    \shf{G}\mapsto\bigl(K\mapsto\fib(\shf{G}(X)\to\shf{G}(K))\bigr)
  \end{equation*}
  is an equivalence
  with the inverse
  \begin{equation*}
    \shf{G}'\mapsto\bigl(K\mapsto\cofib(\shf{G}'(K)\to\shf{G}'(X))\bigr).
  \end{equation*}
  Since
  the equivalence
  \(\Shv(X;\Cat{Sp})\to\cat{C}\) is defined as
  \begin{equation*}
    \shf{F}\mapsto\biggl(K\mapsto\injlim_{U\supset K}\shf{F}(U)\biggr),
  \end{equation*}
  we obtain the desired description of~\(\VD\).
\end{proof}

\begin{remark}\label{xg28w6}
The proof of \cref{sc_verdier} above
  can be regarded as a reduction
  to the spectral case.
  We consider
  the morphism \(p\colon\gamma(X^{\op})\to X^{\op}\)
  and its dual \(q\colon(\gamma(X^{\op}))^{\op}\to X\).
  Via Verdier duality for \(\gamma(X^{\op})^{\op}\),
  we have identified
  the essential image
  of~\(q^{*}\) with that of~\((p_{*})^{\vee}\).
\end{remark}

\begin{remark}\label{xvzpx4}
One can observe that our Verdier duality statement
  comes from sheaf operations:
  For a stably compact space~\(X\),
  we write \(T\) for the closure of the (set-theoretic) diagonal
  of \(X^{\op}\times X\).
  Then the pairing
  \begin{equation*}
    B\colon\Shv(X^{\op};\Cat{Sp})\otimes_{\Cat{Sp}}\Shv(X;\Cat{Sp})\simeq\Shv(X^{\op}\times X;\Cat{Sp})
    \xrightarrow{i^{*}}\Shv(T;\Cat{Sp})\xrightarrow{p_{*}}\Shv(*;\Cat{Sp})
    \simeq\Cat{Sp}
  \end{equation*}
  determines the equivalence in \cref{sc_verdier} via
  \begin{equation*}
    \Shv(X;\Cat{Sp})
    \to[\Shv(X^{\op};\Cat{Sp}),\Shv(X^{\op};\Cat{Sp})\otimes\Shv(X;\Cat{Sp})]
    \xrightarrow{B}[\Shv(X^{\op};\Cat{Sp}),\Cat{Sp}],
  \end{equation*}
  where \([\X,\X]\) denotes
  the internal mapping object in~\(\Cat{Pr}\).
\end{remark}

\section{Continuous spectrum}\label{s:sm_con}

In this section,
we introduce the notion of continuous spectrum,
which is the main object of study in this paper.
We begin in \cref{ss:vsc}
by introducing the notion of very Schwartz coidempotents.
In \cref{ss:sm_con}, we show that they form a stably continuous frame.
In \cref{ss:adj}, we establish the adjoint characterization,
thereby obtaining the unstable generalization of \cref{main_dbl}.

\subsection{Very Schwartz coidempotents}\label{ss:vsc}

We refer the reader to~\cite[Section~2.2]{ttg-sm}
for the notion of coidempotents.

\begin{definition}\label{xdgmmi}
  Let \(\cat{C}\) be an object of \(\CAlg(\Cat{Pr}^{\con})\).
  We consider the idealoid~\(\cat{I}\)
  of~\(\cIdem(\cat{C})\) spanned by
  morphisms of coidempotents \(C\to D\)
  that are compact as morphisms in~\(\cat{C}\).
  An \(\Ind\)-coidempotent of~\(\cat{C}\) is called
  \emph{Schwartz} and \emph{very Schwartz} if
  it is in \(\Ind_{\cat{I}}(\cIdem(\cat{C}))\)
  and \(\Ind_{\cat{I}_{\sd}}(\cIdem(\cat{C}))\)
  (see \cref{xhyy8b} for \((\X)_{\sd}\)),
  respectively.
\end{definition}

\begin{proposition}\label{x8yfz6}
  For \(\cat{C}\in\CAlg(\Cat{Pr}^{\con})\),
  we write \(\Ind(\cIdem(\cat{C}))'\)
  for the full subcategory
  of \(\Ind(\cIdem(\cat{C}))\)
  spanned by very Schwartz \(\Ind\)-coidempotents.
  The functor
  \begin{equation*}
    \Ind(\cIdem(\cat{C}))'
    \to
    \cIdem(\cat{C})
  \end{equation*}
  is fully faithful.
\end{proposition}

\begin{proof}
  From the proof of \cref{con_0}
  we know that for \(\cat{C}\in\Cat{Pr}^{\con}\),
  the functor
  \(\Ind_{\cat{I}}(\cat{C})\to\cat{C}\)
  is an equivalence,
  when \(\cat{I}\) is the idealoid of compact morphisms.
The functor in the statement
  is a full subfunctor of
  the induced equivalence
  \(\Ind_{\cat{I}}(\cat{C})_{/\unit}\to\cat{C}_{/\unit}\).
\end{proof}

\begin{definition}\label{x01624}
  In the situation of \cref{x8yfz6},
  we call a coidempotent
  \emph{very Schwartz}
  if it is inside the essential image of the functor.
  We write \(\cIdem_{\vSc}(\cat{C})\) for this poset.
\end{definition}

\begin{example}\label{xy5x4z}
  In the situation of \cref{x01624},
  let \(C\to\unit\) be a coidempotent object
  such that \(C\) is compact.
  Then it is very Schwartz.
\end{example}

\begin{remark}\label{xa7ror}
  Even in the compactly generated setting
  (i.e., \(\cat{C}\in\CAlg(\Cat{Pr}^{\cg})\)),
  there exist very Schwartz coidempotents
  that are not compact; see \cref{xrxo97}.
\end{remark}

\begin{remark}\label{xxyv58}
  In the situation of \cref{x01624},
  when \(\cat{C}\) is furthermore stable,
  we can equivalently consider idempotents.
  By~\cite[Proposition~2.14]{ttg-sm},
  the assignment
  \((c\colon C\to\unit)\mapsto(\unit\to\cofib(c))\)
  gives an equivalence
  \(\cIdem(\cat{C})\simeq\Idem(\cat{C})\).
  The idealoids of compact morphisms
  on both sides match up
  under this equivalence by \cref{cpt_st}.
  Hence we can similarly define \(\Idem_{\vSc}(\cat{C})\),
  which is equivalent to \(\cIdem_{\vSc}(\cat{C})\).
\end{remark}

\subsection{Continuous spectrum}\label{ss:sm_con}

\begin{proposition}\label{x9ydr9}
  Let \(\cat{C}\) be an object of \(\CAlg(\Cat{Pr}^{\con})\).
  Then \(\cIdem_{\vSc}(\cat{C})\)
  is a stably continuous frame.
\end{proposition}

\begin{remark}\label{xymzpz}
  With some additional effort, one sees directly
  that \(\cIdem_{\vSc}\)
  defines a functor \(\CAlg(\Cat{Pr}^{\con})\to\Cat{SCFrm}\).
  Instead, we see this by proving the adjoint characterization
  in \cref{ss:adj}.
\end{remark}

\begin{definition}\label{xj21xs}
  For an object~\(\cat{C}\) of \(\CAlg(\Cat{Pr}^{\con})\),
  we write \(\Sm^{\con}(\cat{C})\)
  for the stably compact locale corresponding
  to \(\cIdem_{\vSc}(\cat{C})\)
  and call it the \emph{continuous spectrum} of~\(\cat{C}\).
\end{definition}

We now prove \cref{x9ydr9}.
First, we see that it is a frame:

\begin{lemma}\label{xjwh0f}
  In the situation of \cref{x9ydr9},
  the full subposet
  \(\cIdem_{\vSc}(\cat{C})\subset\cIdem(\cat{C})\)
  is a subframe.
\end{lemma}

\begin{proof}
  Recall from~\cite[Theorem~3.8]{ttg-sm}
  that
  we can describe
  the frame operations in
  \(\cIdem(\cat{C})\) as follows:
  \begin{enumerate}
    \item\label{i:f_meet}
      Finite meets are computed as tensor products.
    \item\label{i:d_join}
      Directed joins are computed as filtered colimits.
    \item\label{i:f_join}
      The least element~\(\bot\) is computed as the initial object.
      The join~\(C\vee D\) is computed by taking the pushout of
      \(C\gets C\otimes D\to D\).
  \end{enumerate}
  We need to check that \(\cIdem_{\vSc}(\cat{C})\)
  is closed under these operations.

We first treat \cref{i:f_meet}.
  First, the unit object is very Schwartz,
  since it is compact
  by \cref{x0daz5}.
  When \(C\) and~\(D\) are very Schwartz,
  we lift them to very Schwartz \(\Ind\)-coidempotents
  \((C_{i})_{i}\) and \((D_{j})_{j}\).
  Then \((C_{i}\otimes D_{j})_{i,j}\) is very Schwartz
  by \cref{x0daz5}.
  This shows that \(C\otimes D\) is very Schwartz.
  Similarly, we see that \cref{i:d_join} works
  by taking the corresponding very Schwartz \(\Ind\)-coidempotents
  and using \cref{xm2a0p}.
We then consider~\cref{i:f_join}.
  The initial object is compact and thus very Schwartz.
  For binary joins,
  by arguing as above,
  it suffices to show that
  the \(\Ind\)-coidempotent
  \((C_{i}\vee D_{j})_{i,j}\) is very Schwartz
  when so are \((C_{i})_{i}\) and \((D_{j})_{j}\).
  This follows from
  combining \cref{x0daz5} with \cref{xbagrg}.
\end{proof}

\begin{lemma}\label{xwl07g}
  In the situation of \cref{x9ydr9},
  the relation \(C\leq D\) in \(\cIdem_{\vSc}(\cat{C})\)
  refines to \(C\ll D\)
  if and only if its underlying morphism in~\(\cat{C}\) is compact.
\end{lemma}

\begin{proof}
  We use \cref{xjwh0f} freely in this proof.

  The “if” direction follows from
  the fact that
  the forgetful functor
  \(\cIdem_{\vSc}(\cat{C})\to\cat{C}\)
  preserves filtered colimits.

We prove the “only if” direction.
  We write~\(\cat{I}\) for the idealoid of \(\cIdem(\cat{C})\) in \cref{xdgmmi}.
  Let \(I\) be an object of \(\Ind_{\cat{I}_{\sd}}(\cIdem(\cat{C}))\)
  whose join is~\(D\),
  regarded as an ideal of \(\cIdem(\cat{C})\).
  Consider the intersection \(J=I\cap{\cIdem_{\vSc}(\cat{C})}\).
  We claim that \(J\subset I\) is cofinal.
  Indeed,
  for \(D'\in I\),
  we have a factorization
  \(D'=D'(0)\to D'(1)\to D\)
  with \(D'\colon\QQ\cap[0,1]\to\cIdem(\cat{C})\)
  such that \(D'(p)\to D'(q)\) is compact for \(p<q\).
  We then take
  \begin{equation*}
    D''=\injlim_{i<1/2}D'(i)
  \end{equation*}
  and see that \(D''\) is inside~\(J\).
  This cofinality shows that
  \(D\) can be written as the directed join~\(\bigvee J\)
  in \(\cIdem(\cat{C})\) and hence in \(\cIdem_{\vSc}(\cat{C})\).
  By assumption,
  we can find \(D'\in J\) such that
  \(C\) factors through \(C\to D'\to D\).
  By arguing as above again,
  we obtain a further factorization
  \(C\to D'\to D''\to D\)
  such that \(D'\to D''\) is compact in~\(\cat{C}\).
  Therefore, \(C\to D\) is compact in~\(\cat{C}\).
\end{proof}

\begin{proof}[Proof of \cref{x9ydr9}]
  We write~\(R\) for \(\cIdem_{\vSc}(\cat{C})\),
  which is a frame by \cref{xjwh0f}.
  We then check the conditions in \cref{xbmydj}.
  Since \cref{i:con_alg}
  follows immediately from \cref{xwl07g},
  we are left to verify~\cref{i:con_con}.
  Consider an element \(C\in\cIdem_{\vSc}(\cat{C})\).
  As in the proof of \cref{xwl07g},
  we write \(C\) as the join of~\(I\cap\cIdem_{\vSc}(\cat{C})\),
  where \(I\) is the ideal corresponding
  to the very Schwartz \(\Ind\)-coidempotent
  corresponding to~\(C\).
  By \cref{xwl07g}, we see that \cref{i:con_con} is satisfied.
\end{proof}

\subsection{The continuous sheaves--spectrum adjunction}\label{ss:adj}

We prove the adjunction we have promised:

\begin{theorem}\label{sm_con_adj}
  Let \(\cat{C}\) be an object of
  \(\CAlg(\Cat{Pr}^{\con})\)
  and \(R\) a stably continuous frame.
  Then the map
  \begin{equation*}
    \Map_{\Cat{Frm}}(R,\cIdem_{\vSc}(\cat{C}))
    \hookrightarrow
    \Map_{\Cat{Frm}}(R,\cIdem(\cat{C}))
    \simeq
    \Map_{\CAlg(\Cat{Pr})}(\Shv(R),\cat{C}),
  \end{equation*}
  where the second equivalence is from~\cite[Theorem~B]{ttg-sm},
  restricts to an equivalence
  \begin{equation*}
    \Map_{\Cat{SCFrm}}(R,\cIdem_{\vSc}(\cat{C}))
    \simeq
    \Map_{\CAlg(\Cat{Pr}^{\con})}(\Shv(R),\cat{C}).
  \end{equation*}
\end{theorem}

From this,
we immediately obtain the following,
which is the unstable generalization of \cref{main_dbl}:

\begin{corollary}\label{xdiyc5}
  The continuous spectrum
  \(\Sm^{\con}\)
  defines a functor \(\CAlg(\Cat{Pr}^{\con})\to\Cat{SC}^{\op}\)
  that is right adjoint to~\(\Shv\).
\end{corollary}

\begin{proof}[Proof of \cref{sm_con_adj}]
  Fix a morphism
  \(F\colon\Shv(R)\to\cat{C}\) in \(\CAlg(\Cat{Pr})\).
  This determines a frame morphism
  \(f\colon R\to\cIdem(\cat{C})\)
  by restriction.
  What we have to show is that
  \(F\) is in \(\CAlg(\Cat{Pr}^{\con})\)
  if and only if
  \(f\) restricts to a morphism
  of stably continuous frames \(R\to\cIdem_{\vSc}(\cat{C})\).

  We first prove the “only if” direction.
  By \cref{xcy9f7},
  when \(x\ll x'\),
  the morphism \(F(j(x))\to F(j(x'))\) is compact.
  Therefore,
  for any \(x\in R\),
  the \(\Ind\)-coidempotent \((j(y))_{y\ll x}\)
  is very Schwartz.
  Hence \(j(x)=\injlim_{y\ll x}j(y)\)
  is a very Schwartz coidempotent,
  and \(f\) restricts to a frame morphism
  \(R\to\cIdem_{\vSc}(\cat{C})\).
  Since it also preserves the way-below relation,
  it is also a morphism of stably continuous frames.

  We then prove the “if” direction
  by showing that the right adjoint~\(G\) of~\(F\)
  preserves filtered colimits.
  It suffices to show that the morphism
  \begin{equation*}
    \biggl(\injlim_{i}G(C_{i})\biggr)(x)
    \simeq\projlim_{y\ll x}\injlim_{i}\Map(F(j(y)),C_{i})
    \to\Map(F(j(x)),C)
    \simeq G(C)(x),
  \end{equation*}
  where the first equivalence
  follows from \cref{sheafify},
  is an equivalence.
  By assumption,
  \((F(j(z)))_{z\ll y}\) is very Schwartz,
  so the desired claim follows from \cref{xd0jpr}.
\end{proof}

\section{Computing continuous spectra}\label{s:ex}

In this section,
we perform sample computations
of~\(\Sm^{\con}(\cat{C})\)
for two kinds of~\(\cat{C}\):
In \cref{ss:sm_qcoh},
we consider \(\cat{C}=\D(X)\)
for a scheme~\(X\).
In \cref{ss:sm_sc},
we consider \(\cat{C}=\Shv(X;\Cat{Sp})\)
for a stably compact space~\(X\).

\begin{remark}\label{x2xgqq}
Another potentially interesting computation
  arises in the setting of Clausen–Scholze’s analytic geometry.
  Let \(K\) be a Stein compact set,
  and \(\cat{C}\) denote the rigidification (see \cref{rigidify})
  of \(\D(\Cls{O}(K)_{\liq})\),
  where \(\shf{O}(K)_{\liq}\) is the ring
  of overconvergent holomorphic functions on~\(K\),
  equipped with the liquid analytic structure~\cite{Analytic}
  (with respect to an implicitly fixed parameter \(0 < p \leq 1\)).
  By \cref{ana_rig},
  \(\cat{C}\) can be computed as the \(\infty\)-category
  of modules over~\(\Cls{O}(K)\)
  in the rigidification of \(\D(\CC_{\liq})\).
  Their lecture~\cite{Complex}
  constructs a canonical map
  \(\Sm(\D(\shf{O}(K))) \to \lvert K \rvert\),
  where \(\lvert K \rvert\) denotes the underlying topological space.
Hence, by \cref{main_rig}, we obtain a morphism
  \(\Sm^{\rig}(\cat{C}) \to \lvert K \rvert\)
  of compact Hausdorff spaces.
  However, this map is \emph{not} an isomorphism;
  in particular,
  \(\Sm^{\con}(\cat{C}) \to \lvert K \rvert\)
  cannot be an isomorphism either.
  The space \(\Sm^{\rig}(\cat{C})\)
  admits many additional open subsets
  corresponding to various growth conditions.
\end{remark}

\begin{remark}\label{xxwrfo}
  It would be interesting to study
  the continuous or rigid spectrum
  of the stable presentably symmetric monoidal \(\infty\)-category
  of localizing motives,
  which Efimov~\cite{EfimovRigidity} proved to be rigid.
\end{remark}

\subsection{Quasicoherent sheaves on schemes}\label{ss:sm_qcoh}

We first recall the following:

\begin{definition}\label{xzd2iz}
  Let \(X\) be a quasicompact quasiseparated scheme.
  We write \(\lvert X\rvert\) for the underlying spectral space.
  This induces a canonical map of locales
  \(\Sm(\D(X))\to\lvert X\rvert^{\op}\)
  (see \cref{xoc8bw});
  on the level of frames,
  it maps cocompact closed subset \(X\setminus U\)
  to \(j_{*}\Cls{O}_{U}\),
  where \(j\) is the inclusion of \(U\) into \(X\).
  We say that
  \(X\) satisfies the \emph{telescope conjecture}
  if this map is an equivalence.
\end{definition}

The following was proven
in the affine case in~\cite{Neeman92}
and in general in~\cite{ATJLSS04}:

\begin{theorem}[Neeman, Alonso~Tarrío–Jeremías~López–Souto~Salorio]\label{x4r6i1}
  Any noetherian scheme
  satisfies the telescope conjecture.
\end{theorem}

\begin{remark}\label{x79j7u}
  In the situation of \cref{xzd2iz},
  whether \(X\) satisfies
  the telescope conjecture can be checked at the level of local rings
  by~\cite[Theorem~4.9]{HrbekHuZhu24}.
  For example,
  if all local rings are noetherian,
  \(X\) satisfies the telescope conjecture.
\end{remark}

We see that \(\Sm^{\con}\) is not interesting in this case:

\begin{theorem}\label{xgbjvz}
  Let \(X\) be a noetherian scheme.
Then \(\Sm^{\con}(\D(X))\) is the (finite) set of connected components.
\end{theorem}

\begin{remark}\label{xuh4yi}
We can equivalently rephrase \cref{xgbjvz}
  that
  \(\Sm^{\con}(\D(\X))\) recovers its Pierce spectrum
  (see \cref{xuujx5})
  for noetherian schemes.
  This is not true in general
  without noetherianness,
  as we see in \cref{xrxo97}.
\end{remark}

\begin{lemma}\label{xb6h9a}
  Let \(X\) be a noetherian scheme
  and \(U'\subset U\) an inclusion
  of quasicompact open subsets.
  We have corresponding idempotent algebras
  \(j_{*}\Cls{O}_{U}\) and \(j'_{*}\Cls{O}_{U'}\)
  in~\(\D(X)\).
  If the map \(j_{*}\Cls{O}_{U}\to j'_{*}\Cls{O}_{U'}\)
  factors through a perfect complex,
  there is a closed subset~\(Z\)
  such that \(U'\subset Z\subset U\) holds.
\end{lemma}

\begin{remark}\label{x5g2jo}
Note that \cref{xb6h9a} cannot be generalized
  to arbitrary quasicompact quasiseparated schemes.
  Consider a valuation ring~\(V\)
  with nonzero noninvertible elements~\(x\) and~\(x'\)
  satisfying \(x^{n}\mid x'\) for any~\(n\).
  Then the map \(V[1/x]\to V[1/x']\)
  factors through a perfect complex.
\end{remark}

\begin{proof}
Suppose the contrary.
  This means that there is a point \(x\notin U\)
  that is in the closure of~\(U'\).
  Choose
  a discrete valuation ring~\(V\)
  and a commutative diagram
  \begin{equation*}
    \begin{tikzcd}
      \Spec(\Frac(V))\ar[r]\ar[d,hook]&
      U'\ar[d,hook]\\
      \Spec V\ar[r]&
      X\rlap,
    \end{tikzcd}
  \end{equation*}
  where the bottom map maps
  the closed point to~\(x\).
  We write \(E\) and~\(E'\)
  for the base change of the idempotent algebras
  \(j_{*}\Cls{O}_{U}\) and \(j'_{*}\Cls{O}_{U'}\),
  respectively.
We have a sequence of idempotent algebras
  \(V\to E\to E'\to\Frac(V)\)
  such that \(V\to E\) is not an equivalence.
  This forces \(E=E'={\Frac(V)}\),
  but it is not perfect, which is a contradiction.
\end{proof}

\begin{proof}[Proof of \cref{xgbjvz}]
  By considering connected components,
  we can assume that \(X\) is connected.
  In this case,
  we need to show that a very Schwartz idempotent
  is~\(\Cls{O}_{X}\) or \(0\).

  By \cref{x4r6i1},
  any idempotent is represented by a \(\Pro\)-object
  of quasicompact subsets \((U_{i})_{i}\).
  When an inclusion \(\emptyset\neq U'\subset U\neq X\)
  determines a compact morphism,
  by \cref{xb6h9a}
  and the connectedness assumption,
  it must be strict.
Hence, if there is a nontrivial very Schwartz idempotent,
  we can find an ascending sequence of open subsets,
  which contradicts
  the noetherianness assumption.
\end{proof}

We finish this section with
some sporadic observations in the non-noetherian case:

\begin{example}\label{xrxo97}
  This example is based on \cref{x5g2jo}.
Let \(k\) be a field.
  We take \(P=\QQ\cap[0,1)\)
  and consider \(\Gamma=\bigoplus_{p\in P}\ZZ\)
  with the lexicographic order,
  which becomes a totally ordered abelian group.
  Then we take~\(V\) to
  be the ring of Hahn series \(k\llbracket T^{\Gamma_{\geq0}}\rrbracket\).
  We can see that
  for \(p<p'\) in~\(P\),
  the map
  \(V[T^{-p}]\to V[T^{-p'}]\) is compact,
  since it factors through \(T^{-p'}V\).
  Therefore,
  \(\Frac(V)=\injlim_{p}V[T^{-p}]\) is a very Schwartz idempotent.
  See also \cref{xu0zbv}.
\end{example}

\begin{example}\label{xf5r8m}
  Let \(A\) be a Boolean ring.
  Then \(\Spec A\)
  satisfies the telescope conjecture,
  e.g., by \cref{x79j7u}.
  In this case,
  all idempotents are very Schwartz,
  and hence it recovers
  the Zariski (or equivalently, Pierce)
  spectrum of~\(A\).
\end{example}

\subsection{Spectral sheaves on stably compact spaces}\label{ss:sm_sc}

We here prove \cref{tan_sc}:

\begin{theorem}\label{xw2z2y}
  For a stably compact space~\(X\),
  the morphism
  \(\Sm^{\con}(\Shv(X;\Cat{Sp}))\to X\)
  is an equivalence.
\end{theorem}

We start with the following observation:

\begin{proposition}\label{xvs1r4}
  Let \(p\) be a prime.
  If \((E_{i})_{i}\)
  is a Schwartz \(\Ind\)-idempotent
  of~\(\Cat{Sp}_{(p)}\),
  it is equivalent to
  either~\(\SS_{(p)}\) or~\(0\).
\end{proposition}

To prove \cref{xvs1r4},
we need a simple observation
from chromatic homotopy theory.
For a prime~\(p\) (which is implicit in the notation),
we write \(L_{n}\) for the Bousfield localization
with respect to~\(E(n)\)
and \(L_{n}^{\fin}\)
the corresponding finite localization
(i.e., the Bousfield localization
with respect to~\(T(0)\oplus\dotsb\oplus T(n)\)).

\begin{lemma}\label{xadi4q}
  For \(n>0\),
  the map
  \(L_{n}^{\fin}\SS
  \to L_{n}\SS\)
  does not factor through a spectrum bounded below.
\end{lemma}

\begin{proof}
Suppose that it factors through a bounded-below spectrum~\(C\).
  Let \(F\) be a (nonzero) finite \(p\)-local spectrum
  of type~\(n\)
  and \(v_{n}\colon\Sigma^{d}F\to F\) a \(v_{n}\)-self map.
  Then
  \(L_{n}^{\fin}F
  =v_{n}^{-1}F\) is \(v_{n}\)-periodic.
  So
  the map \(L_{n}^{\fin}F\to L_{n}F\)
  factors through
  the limit
  \begin{equation*}
    \dotsb\xrightarrow{\Sigma^{d}v_{n}\otimes{\id}}\Sigma^{d}F\otimes C\xrightarrow{v_{n}\otimes{\id}}F\otimes C,
  \end{equation*}
  which vanishes, since \(C\) is bounded below.
Therefore, the map should be zero,
  which is contradicting
  the fact that \(F\to L_{K(n)}F\) is nonzero.
\end{proof}

\begin{proof}[Proof of \cref{xvs1r4}]
Suppose that it is neither~\(\SS_{(p)}\) nor~\(0\).
  We assume that the index \(\infty\)-category is a poset.
  For each~\(i\),
  Then there is a (unique) \(n\geq0\)
  such that
  \(E_{i}\) fits between~\(L_{n}^{\fin}\SS\)
  and~\(L_{n}\SS\).
  Therefore,
  we can find a factorization
  \(L_{n}^{\fin}\SS\to E_{i}\to E_{j}\to L_{n}\SS\)
  for some~\(n\) and \(i\leq j\)
  such that \(E_{i}\to E_{j}\) is compact,
  and hence it factors through a spectrum bounded below by \cref{xo7ubc}.
  When \(n=0\), this is impossible,
  since \(\QQ\) is not compact.
  When \(n>0\),
  this is impossible by \cref{xadi4q}.
\end{proof}

\begin{corollary}\label{xaymiy}
  The space
  \(\Sm^{\con}(\Cat{Sp})\)
  is a singleton.
\end{corollary}

\begin{proof}
  Let \(E\) be a very Schwartz idempotent.
  Our aim is to show that it is equivalent
  to~\(\SS\) or~\(0\).
  Since \(E\otimes\QQ\) is an idempotent
  in~\(\D(\QQ)\), it is either~\(\QQ\) or~\(0\).
  If \(E\otimes\QQ\) vanishes,
  by \cref{xvs1r4},
  \(E\) itself must vanish.
We then assume
  \(E\otimes\QQ\simeq\QQ\).
  By \cref{xvs1r4},
  it is of the form \(\SS[1/P]\)
  for some set of primes~\(P\).
  In this case,
  we base change to~\(\ZZ\)
  to see that \(P\) must be empty
  by \cref{xgbjvz}.
\end{proof}

\begin{proof}[Proof of \cref{tan_sc}]
  We know that \(\SS_{Z}\) is a very Schwartz idempotent
  for a closed subset~\(Z\);
  see the proof of \cref{sm_con_adj}.
  We prove the converse.

  Let \(E\) be a very Schwartz idempotent.
  By \cref{xaymiy},
  \(E_{x}\) is either~\(\SS\) or~\(0\)
  for any point~\(x\).
  Let \(U\) be the subset of~\(X\)
  spanned by points~\(x\)
  such that \(E_{x}\) vanishes.
Since being zero for a ring
  is equivalent to \(0=1\),
  we see that \(U\) is open.
  We then replace~\(X\) with the complement of~\(U\).
  Now,
  it suffices to show that
  for any very Schwartz idempotent~\(E\) on~\(X\),
  the map \(\SS\to E\) is an equivalence
  under the assumption that
  it is a stalkwise equivalence.

  We write \(E\) as the colimit of
  a very Schwartz \(\Ind\)-idempotent \((E_{i})_{i}\).
  We fix~\(i\) and show \(E_{i}=\SS\).
For each point~\(x\),
  we get a map
  \(E\to\SS_{\overline{\{x\}}}\)
  of idempotent algebras,
  since \(E_{x}\simeq\SS\) by assumption.
  Since \(E_{i}\to E\) is compact,
  we obtain a factorization
  \(E_{i}\to\SS_{Z_{x}}\to\SS_{\overline{\{x\}}}\)
  for some closed subset~\(Z_{x}\)
  satisfying
  \(X\setminus Z_{x}\ll X\setminus\overline{\{x\}}\).
Therefore,
  we can take a compact saturated subset~\(K_{x}\)
  such that
  \(X\setminus Z_{x}\subset K_{x}\subset X\setminus\overline{\{x\}}\) holds.
  Since \(X^{\op}\) is compact,
  we see that there exist
  \(x_{1}\), …,~\(x_{n}\)
  such that \(X\setminus K_{x_{1}}\),
  …,~\(X\setminus K_{x_{n}}\) cover~\(X\).
  Hence \(Z_{1}\), …,~\(Z_{n}\) also cover~\(X\),
  and we indeed obtain the desired map \(E_{i}\to\SS\).
\end{proof}

\section{Rigid spectrum}\label{s:rig_sp}

We now consider the rigid spectrum~\(\Sm^{\rig}\),
which, as noted in \cref{ss:tan_ch},
is simply the Stone–Čech compactification of~\(\Sm\).
Here, we provide a concrete definition
in terms of very nuclear idempotents.
In \cref{ss:tc,ss:rig},
we recall the notion of rigid
stable presentably symmetric monoidal \(\infty\)-categories.
This sets the stage for \cref{ss:nr},
where we introduce the rigid spectrum for such categories.
We then compute some examples
in \cref{ss:rig_sample}.
Finally, in \cref{ss:rig_thm},
we digress to highlight the role of very nuclear algebras in general.

\subsection{Trace-class morphisms and (very) nuclear \texorpdfstring{\(\Ind\)}{Ind}-objects}\label{ss:tc}

The notion of nuclear modules
in the abstract categorical setting
was first systematically studied
by Clausen–Scholze (cf.~\cite[Section~13]{Analytic}).
We here review basics of this theory.

We first recall the following:

\begin{definition}\label{xz1vid}
  We say that
  an object~\(C\) in a symmetric monoidal \(\infty\)-category
  is \emph{dualizable} if there are an object~\(D\)
  and morphisms
  \(\eta\colon\unit\to D\otimes C\)
  and
  \(\epsilon\colon C\otimes D\to\unit\)
  such that
  \begin{align*}
    C\simeq
    C\otimes\unit&\xrightarrow{{\id_{C}}\otimes\eta}
    C\otimes D\otimes C\xrightarrow{\epsilon\otimes{\id_{C}}}
    \unit\otimes C
    \simeq C,\\
    D\simeq
    \unit\otimes D&\xrightarrow{{\id_{D}}\otimes\eta}
    D\otimes C\otimes D\xrightarrow{\epsilon\otimes{\id_{D}}}
    D\otimes\unit
    \simeq D,
  \end{align*}
  are equivalent to \({\id}_C\) and \({\id}_D\), respectively.
\end{definition}

A modification leads us to the following definition:

\begin{definition}\label{xi691o}
  We say that a morphism \(f\colon C\to C'\)
  in a symmetric monoidal \(\infty\)-category is \emph{trace class} if
  there are
  morphisms \(\eta\colon\unit\to D\otimes C'\)
  and \(\epsilon\colon C\otimes D\to\unit\)
  such that the composite
  \begin{equation*}
    C\simeq
    C\otimes\unit\xrightarrow{{\id_{C}}\otimes\eta}
    C\otimes D\otimes C'\xrightarrow{\epsilon\otimes{\id_{C'}}}
    \unit\otimes C'
    \simeq C'
  \end{equation*}
  is equivalent to~\(f\).
\end{definition}

\begin{example}\label{xajqb1}
  For \(\cat{C}\in\CAlg(\Cat{Pr})\),
  when the unit is compact,
  any trace-class morphisms \(C\to C'\) is compact.
  Indeed,
  adopting the notation from \cref{xi691o},
  the natural transformation
  \([C',\X]\to[C,\X]\)
  factors through \(D\otimes\X\),
  where \([\X,\X]\) denotes the internal mapping object.
  Consider a filtered colimit \(E=\injlim_{k}E_{k}\).
  \begin{equation*}
    \begin{tikzcd}
      \injlim_{k}[C',E_{k}]\ar[r]\ar[d]&
      \injlim_{k}(D\otimes E_{k})\ar[r]\ar[d,"\simeq"]&
      \injlim_{k}[C,E_{k}]\ar[d]\\
      {\bigl[C',\injlim_{k}E_{k}\bigr]}\ar[r]&
      D\otimes\bigl(\injlim_{k}E_{k}\bigr)\ar[r]&
      {\bigl[C,\injlim_{k}E_{k}\bigr]}
    \end{tikzcd}
  \end{equation*}
  By applying \(\Map(\unit,\X)\)
  and using the compactness of~\(\unit\),
  we get the desired lift.

  We see that the converse holds
  when \(\cat{C}\) is stable
  and moreover rigid; see \cref{rig_tc}.
\end{example}

\begin{example}\label{x4psyd}
  For \(\cat{C}\in\CAlg(\Cat{Pr})\),
  suppose that the unit is \(\kappa\)-compact
  and the underlying \(\infty\)-category
  is \(\kappa\)-compactly generated.
  Then any trace-class morphism
  factors through a \(\kappa\)-compact object.
  Indeed,
  with the notation of \cref{xi691o},
  the unit
  \(\eta\) factors through some \(C'_{0}\to C'\),
  where \(C'_{0}\) is \(\kappa\)-compact.
\end{example}

\begin{remark}\label{xa0g0m}
  In \cref{xi691o},
  when the symmetric monoidal structure is closed,
  we can always take \(D\) to be \([C,\unit]\),
  where \([\X,\X]\) denotes
  the internal mapping object.
  In this case,
  a morphism \(C\to C'\) is trace class
  if and only if
  the corresponding morphism
  \(\unit\to[C,C']\)
  factors through \([C,\unit]\otimes C'\to[C,C']\).
  This aligns with the intuition
  behind trace-class operators in functional analysis;
  recall that
  a bounded operator on~\(\ell^{2}\)
  is trace class if
  it lies in the image of
  \((\ell^{2})^{*}\otimes_{\pi}\ell^{2}\to\Cls{B}(\ell^{2})\).
\end{remark}

We note some formal properties of trace-class morphisms:

\begin{lemma}\label{xd9v6a}
  In a symmetric monoidal \(\infty\)-category,
  trace-class morphisms form
  an idealoid (see \cref{xbjisd}).
\end{lemma}

\begin{proof}
  This immediately follows from the definition.
\end{proof}

\begin{lemma}\label{xjd6rf}
  We consider a diagram
  \begin{equation*}
    \begin{tikzcd}
      C_{0}'\ar[r]\ar[d]&
      C_{1}'\ar[r]\ar[d]&
      C_{2}'\ar[d]\\
      C_{0}\ar[r]&
      C_{1}\ar[r]&
      C_{2}
    \end{tikzcd}
  \end{equation*}
  in a stably symmetric monoidal \(\infty\)-category~\(\cat{C}\).
  We write
  \(C_{0}''\to C_{1}''\to C_{2}''\)
  for the cofibers of the vertical morphisms.
  If \(C_{0}'\to C_{1}'\)
  and \(C_{1}\to C_{2}\) are trace class,
  so is \(C_{0}''\to C_{2}''\).
\end{lemma}

\begin{proof}
One could prove this directly,
  but we opt here for a shortcut
  derived from \cref{xa0g0m}.
  We assume that \(\cat{C}\) is small
  and consider \(\Ind(\cat{C})\),
  for which we can use \cref{xa0g0m}.
  Then we claim that a morphism
  \(C\to C'\) in~\(\cat{C}\)
  is trace class in~\(\cat{C}\)
  if and only if so in~\(\Ind(\cat{C})\).
  The “only if” direction is clear,
  and the “if” direction follows from
  the compactness of~\(\unit\in\Ind(\cat{C})\):
  We write \([C,\unit]\) as the colimit
  \(\injlim_{i}D_{i}\)
  of some filtered diagram in~\(\cat{C}\),
  but then a witness \(\unit\to[C,\unit]\otimes C'\)
  factors through \(\unit\to D_{i}\otimes C'\)
  for some~\(i\).

  Therefore,
  we can assume that \(\cat{C}\) is closed
  and write \([\X,\X]\) for the internal mapping object.
  We write \(D_{i}\), \(D_{i}'\), and \(D_{i}''\)
  for \([\X,\unit]\) of \(C_{i}\), \(C_{i}'\), and \(C_{i}''\),
  respectively.
  We fix witnesses
  \begin{align*}
    \eta\colon\unit&\to D_{1}\otimes C_{2},&
    \eta'\colon\unit&\to D_{0}'\otimes C_{1}'.
  \end{align*}
  We then construct
  a morphism \(\eta''\colon\unit\to D_{0}''\otimes C_{2}''\)
  using the diagram
  \begin{equation*}
    \begin{tikzcd}
      \unit\ar[r,dashed]\ar[d,"\eta"']&
      D_{0}''\otimes C_{2}''\ar[r]\ar[d]&
      0\ar[d]\\
      D_{1}\otimes C_{2}\ar[r]&
      D_{0}\otimes C_{2}''\ar[r]&
      D_{0}'\otimes C_{2}''\rlap.
    \end{tikzcd}
  \end{equation*}
  Since the right square is cartesian,
  we need to provide a nullhomotopy for
  the composite
  \(\unit\to D_{1}\otimes C_{2}\to D_{0}'\otimes C_{2}''\).
  This comes from the diagram
  \begin{equation*}
    \begin{tikzcd}
      \unit\ar[r,"\eta'"]\ar[d,"\eta"']&
      D_{0}'\otimes C_{1}'\ar[r]\ar[d]\ar[rr,bend left=15,"0"]&
      D_{0}'\otimes C_{1}\ar[d]\ar[r]\ar[rd]&
      D_{0}'\otimes C_{1}''\ar[rd]&
      \\
      D_{1}\otimes C_{2}\ar[r]&
      D_{0}'\otimes C_{1}'\otimes D_{1}\otimes C_{2}\ar[r]&
      D_{0}'\otimes C_{1}\otimes D_{1}\otimes C_{2}\ar[r]&
      D_{0}'\otimes C_{2}\ar[r]&
      D_{0}'\otimes C_{2}''\rlap.
    \end{tikzcd}
  \end{equation*}
  We then show that this indeed witnesses
  \(C_{0}''\to C_{2}''\).
  To prove this,
  we just need to observe that
  the tautological diagrams
  \begin{equation*}
    \begin{tikzcd}
      \unit\ar[r]\ar[d]&
      {[C_{0}'',C_{2}'']}\ar[r]\ar[d]&
      0\ar[d]\\
      {[C_{1},C_{2}]}\ar[r]&
      {[C_{0},C_{2}'']}\ar[r]&
      {[C_{0}',C_{2}'']}
    \end{tikzcd}
  \end{equation*}
  and
  \begin{equation*}
    \begin{tikzcd}
      \unit\ar[r]\ar[d]&
      {[C_{0}',C_{1}']}\ar[r]\ar[d]\ar[rr,bend left=15,"0"]&
      {[C_{0}',C_{1}]}\ar[d]\ar[r]\ar[rd]&
      {[C_{0}',C_{1}'']}\ar[rd]&
      \\
      {[C_{1},C_{2}]}\ar[r]&
      {[C_{0}',C_{1}']}\otimes{[C_{1},C_{2}]}\ar[r]&
      {[C_{0}',C_{1}]}\otimes{[C_{1},C_{2}]}\ar[r]&
      {[C_{0}',C_{2}]}\ar[r]&
      {[C_{0}',C_{2}'']}
    \end{tikzcd}
  \end{equation*}
  receive maps from the diagrams above,
  respectively.
\end{proof}

\begin{definition}\label{xombx7}
  In a symmetric monoidal \(\infty\)-category,
  we say that a morphism is \emph{very trace class}
  if it is in \(\cat{I}_{\sd}\) (see \cref{xhyy8b}),
  where \(\cat{I}\) is the idealoid of trace-class morphisms.
\end{definition}

\begin{definition}\label{xsdc04}
  Let \(\cat{C}\) be a presentably symmetric monoidal \(\infty\)-category.
  We write \(\cat{I}\) for the class of trace-class morphisms in~\(\cat{C}\).
  We say that an \(\Ind\)-object is
  \emph{nuclear} and \emph{very nuclear}
  if it is in
  \(\Ind_{\cat{I}}(\cat{C})\) (see \cref{strict})
  and \(\Ind_{\cat{I}_{\sd}}(\cat{C})\),
  respectively.
\end{definition}

\begin{remark}\label{xw917h}
  As noted in \cref{x17h0a},
  the functional-analytic analog of
  \cref{xsdc04} is called a “dual nuclear” (or “conuclear”) module.
\end{remark}

In the compactly generated situation,
it has an easier characterization:

\begin{proposition}\label{xzlma5}
  For \(\cat{C}\in\CAlg(\Cat{Pr}^{\cg})\),
  the following are equivalent for an object \(C\in\cat{C}\):
  \begin{conenum}
    \item\label{i:vyhu6}
      It is the colimit
      of some nuclear \(\Ind\)-object.
    \item\label{i:2hh2s}
      For any compact object~\(C_{0}\),
      the map
      \([C_{0},\unit]\otimes C\to[C_{0},C]\)
      in~\(\cat{C}\)
      is an equivalence,
      where \([\X,\X]\) denotes
      the internal mapping object.
  \end{conenum}
\end{proposition}

\begin{proof}
  We first show
  \(\text{\cref{i:vyhu6}}
  \Rightarrow\text{\cref{i:2hh2s}}\).
  Fix a compact object~\(C_{0}\).
  By \cref{xsoyii}
  and the fact that \([C_{0},\X]\)
  preserves filtered colimits,
  it suffices to construct a filler in
  \begin{equation*}
    \begin{tikzcd}
      {[C_{0},\unit]}\otimes C\ar[r]\ar[d]&
      {[C_{0},\unit]}\otimes C'\ar[d]\\
      {[C_{0},C]}\ar[r]\ar[ur,dashed]&
      {[C_{0},C']}
    \end{tikzcd}
  \end{equation*}
  for any trace-class morphism \(C\to C'\),
  but this exists by definition
  (for any \(C_{0}\in\cat{C}\)).

  We then assume
  that \cref{i:2hh2s} is satisfied
  for an object \(C\in\cat{C}\).
  We write \(C\) as a directed colimit \(\injlim_{i}C_{i}\),
  where \(C_{i}\) is compact.
  Then we claim that it is nuclear.
  We take~\(i_{0}\) and consider
  the canonical morphism \(C_{i_{0}}\to C\).
  Then by assumption,
  this is witnessed by a morphism
  \begin{equation*}
    \unit\to[C_{i_{0}},\unit]\otimes C
    \simeq[C_{i_{0}},\unit]\otimes\biggl(\injlim_{i}C_{i}\biggr)
    \simeq\injlim_{i}([C_{i_{0}},\unit]\otimes C_{i}).
  \end{equation*}
  Since \(\unit\) is compact,
  there is \(i\geq i_{0}\) such that
  the map factors through
  \([C_{i_{0}},\unit]\otimes C_{i}\).
  This implies that \(C_{i_{0}}\to C\)
  factors through as \(C_{i_{0}}\to C_{i}\to C\)
  and \(C_{i_{0}}\to C_{i}\) is trace class.
\end{proof}

\begin{remark}\label{ana_cg}
  Beware that
  our assumption in \cref{xzlma5} is \emph{not} satisfied
  in general
  by the \(\infty\)-category of modules over an analytic ring
  (both in light condensed mathematics
  and heavy condensed mathematics with cutoff strong-limit cardinal).
  This is why~\cite[Definition~13.10]{Analytic}
  only considered an equivalence of mapping spectra
  instead of internal mapping objects as in~\cref{i:2hh2s}.
  Nevertheless, this assumption holds in the solid case,
  significantly simplifying the theory.
  See also \cref{ana_rig}.
\end{remark}

\subsection{Rigid categories}\label{ss:rig}

The following was introduced in~\cite[Appendix~D]{Gaitsgory15}:

\begin{definition}[Gaitsgory]\label{xsl2l2}
  A stable presentably symmetric monoidal \(\infty\)-category~\(\cat{C}\)
  is rigid if \(\unit\) is compact
  and the multiplication \(\cat{C}\otimes\cat{C}\to\cat{C}\)
  has a \(\cat{C}\otimes\cat{C}\)-linear right adjoint.
\end{definition}

\begin{example}\label{xxw64m}
  In the situation of \cref{xsl2l2}
  further assume that \(\cat{C}\) is compactly generated.
  Then it is rigid
  if and only if it is rigid in the classical sense;
  i.e., \(\unit\) is compact, and all compact objects are dualizable.
\end{example}

We here explain the following,
which are by now standard;
see~\cite{Ramzi2} for a more general discussion
over bases other than~\(\Cat{Sp}\)
(but see \cref{xu26r5}):

\begin{proposition}\label{rig_0}
  A stable presentably symmetric monoidal \(\infty\)-category~\(\cat{C}\)
  with compact unit
  is rigid
  if and only if every object
  can be written as the colimit
  of a nuclear \(\Ind\)-object.
\end{proposition}

\begin{remark}\label{xu26r5}
  In the literature,
  one typically considers a variant of \cref{rig_0}
  where only “basic” (i.e., sequential) nuclear objects are treated,
  and the \(\infty\)-category is then generated under colimits
  by them.
  Here, by contrast,
  we work with general \(\Ind\)-objects.
  See also \cref{xah8cd}.
\end{remark}

We first prove the following:

\begin{proposition}\label{rig_tc}
  A morphism
  in a rigid stable presentably symmetric monoidal \(\infty\)-category~\(\cat{C}\)
  is compact if and only if it is trace class.
\end{proposition}

\begin{lemma}\label{x7qlp8}
  Consider \(\cat{C}\in\CAlg(\Cat{Pr}_{\st})\)
  and \(\cat{D}\in\Cat{Pr}_{\st}\).
  If \(D\in\cat{D}\) is compact,
  the morphism
  \begin{equation*}
    C\otimes[\unit\boxtimes D,X]
    \to
    [\unit\boxtimes D,C\otimes X]
  \end{equation*}
  is an equivalence
  for any \(C\in\cat{C}\) and \(X\in\cat{C}\otimes\cat{D}\).
  Here for a \(\cat{C}\)-linear \(\infty\)-category~\(\cat{M}\),
  we write \([\X,\X]\) for the mapping object in~\(\cat{C}\);
  i.e., for \(M\) and \(M'\in\cat{M}\),
  the object \([M,M']\in\cat{C}\)
  corepresents the functor \(\Map_{\cat{M}}(\X\otimes M,M')\colon
  \cat{C}^{\op}\to\Cat{S}\).
\end{lemma}

\begin{proof}
  The object~\(D\) corresponds to
  a morphism \(\Cat{Sp}\to\cat{D}\) in~\(\Cat{Pr}_{\st}\)
  and its compactness corresponds to
  the fact that it admits a right adjoint in~\(\Cat{Pr}_{\st}\).
  By applying \(\cat{C}\otimes\X\),
  it shows that
  the \(\cat{C}\)-linear functor
  \(\cat{C}\to\cat{C}\otimes\cat{D}\)
  corresponding to \(\unit\boxtimes D\)
  admits a \(\cat{C}\)-linear right adjoint,
  which implies the desired result.
\end{proof}

\begin{proof}[Proof of \cref{rig_tc}]
  Since the unit is compact,
  trace-class morphisms are compact
  by \cref{xajqb1}.
  Thus, it suffices to prove the converse.
  In this proof,
  we keep the notation of \cref{x7qlp8}.
We consider \(\cat{C}\otimes\cat{C}\)
  as a \(\cat{C}\)-linear \(\infty\)-category
  via the left factor.
  Let \(C\to C'\) be a compact morphism.
  We wish to provide a filler in the diagram
  \begin{equation*}
    \begin{tikzcd}
      {C'\otimes[C',\unit]}\ar[r]\ar[d]&
      {C'\otimes[C,\unit]}\ar[d]\\
      {[C',C']}\ar[r]\ar[ur,dashed]&
      {[C,C']}\rlap.
    \end{tikzcd}
  \end{equation*}
  Since the multiplication
  \(m\colon\cat{C}\otimes\cat{C}\to\cat{C}\)
  admits a \(\cat{C}\)-linear right adjoint~\(m^{\R}\),
  we can equivalently consider the diagram
  \begin{equation*}
    \begin{tikzcd}
      {C'\otimes[\unit\boxtimes C',m^{\R}(\unit)]}\ar[r]\ar[d]&
      {C'\otimes[\unit\boxtimes C,m^{\R}(\unit)]}\ar[d]\\
      {[\unit\boxtimes C',m^{\R}(C')]}\ar[r]\ar[ur,dashed]&
      {[\unit\boxtimes C,m^{\R}(C')]}\rlap.
    \end{tikzcd}
  \end{equation*}
  Take a sufficiently large regular cardinal~\(\kappa\)
  and consider
  \(\check{\jmath}\colon\cat{C}\hookrightarrow\Ind(\cat{C}_{\kappa})\).
  Since \(C\to C'\) is compact,
  \(\check{\jmath}(C)\to\check{\jmath}(C')\)
  factors through~\(j(D)\)
  for some \(D\in\cat{C}_{\kappa}\).
  Now, we obtain the refinement
  \begin{equation*}
    \begin{tikzcd}
      {C'\otimes[\unit\boxtimes C',m^{\R}(\unit)]}\ar[r]\ar[d]&
      {C'\otimes[\unit\boxtimes j(C),({\id}\otimes\check{\jmath})(m^{\R}(\unit))]}\ar[r]\ar[d]&
      {C'\otimes[\unit\boxtimes C,m^{\R}(\unit)]}\ar[d]\\
      {[\unit\boxtimes C',m^{\R}(C')]}\ar[r]&
      {[\unit\boxtimes j(C),({\id}\otimes\check{\jmath})(m^{\R}(C'))]}\ar[r]&
      {[\unit\boxtimes C,m^{\R}(C')]}\rlap,
    \end{tikzcd}
  \end{equation*}
  where the middle vertical arrow
  is an equivalence by \cref{x7qlp8}.
\end{proof}

Now, we are left to prove \cref{rig_0}.
We first observe the following:

\begin{lemma}\label{x01b14}
  The forgetful functor
  \(\CAlg(\Cat{Pr}_{\st})_{\rig}\hookrightarrow\CAlg(\Cat{Pr}_{\st})\)
  factors through \(\CAlg(\Cat{Dbl})\).
\end{lemma}

\begin{proof}
  It suffices to show that
  the forgetful functor
  \(\CAlg(\Cat{Pr}_{\st})_{\rig}\to\Cat{Pr}_{\st}\)
  factors through \(\Cat{Dbl}\).
  For an object~\(\cat{C}\),
  we see that
  \(\unit\xrightarrow{u}\cat{C}\xrightarrow{m^{\R}}\cat{C}\otimes\cat{C}\)
  and \(\cat{C}\otimes\cat{C}\xrightarrow{m}\cat{C}\xrightarrow{u^{\R}}\cat{C}\)
  are the unit and counit in~\(\Cat{Pr}_{\st}\).
  For a morphism \(\cat{C}\to\cat{D}\),
  it can be written as the composite
  \(\cat{C}\to\cat{C}\otimes\cat{D}\to\cat{D}\).
  The first arrow
  is a base change of \(\unit\to\cat{D}\)
  and thus admits a right adjoint in~\(\Cat{Pr}_{\st}\).
  The second arrow
  is a base change of \(\cat{C}\otimes\cat{C}\to\cat{C}\)
  and thus admits a right adjoint in~\(\Cat{Pr}_{\st}\).
\end{proof}

\begin{proof}[Proof of \cref{rig_0}]
  We prove the “only if” direction.
  Suppose that \(\cat{C}\) is rigid.
  Then it is dualizable by \cref{x01b14}
  and hence by \cref{con_0},
  we can write any object
  as a Schwartz \(\Ind\)-object.
  Hence the desired result follows from
  \cref{rig_tc}.

  We then prove the “if” direction.
  Since the unit is compact,
  trace-class morphisms are compact by \cref{xajqb1}.
  Also, since when \(C\to C'\) and \(D\to D'\) are trace class
  then so is \(C\otimes D\to C'\otimes D'\),
  we see that \(\cat{C}\) is in \(\CAlg(\Cat{Dbl})\)
  by \cref{x0daz5}.
  In particular,
  the right adjoint~\(m^{\R}\) of
  \(m\colon\cat{C}\otimes\cat{C}\to\cat{C}\)
  exists in~\(\Cat{Pr}\).
  It suffices
  to show that it is \(\cat{C}\otimes\cat{C}\)-linear.
  We regard \(\cat{C}\otimes\cat{C}\)
  as a \(\cat{C}\)-linear \(\infty\)-category
  via the left factor
  and prove that \(m^{\R}\) is \(\cat{C}\)-linear.
  Let \(X\) be an object of \(\cat{C}\otimes\cat{C}\).
  When \(C\to C'\) is trace class,
  we have a dashed arrow—canonical
  with respect to the witness—in
  the diagram
  \begin{equation*}
    \begin{tikzcd}
      C\otimes m^{\R}(X)\ar[r]\ar[d]&
      C'\otimes m^{\R}(X)\ar[d]\\
      m^{\R}(C\otimes X)\ar[r]\ar[ur,dashed]&
      m^{\R}(C'\otimes X)\rlap.
    \end{tikzcd}
  \end{equation*}
  Using this and \cref{xsoyii},
  the linearity for~\(C\in\cat{C}\) follows
  from writing~\(C\) as the colimit of some nuclear \(\Ind\)-object.
\end{proof}

\begin{corollary}\label{rigidify}
  The inclusion \(\CAlg(\Cat{Pr}_{\st})_{\rig}
  \hookrightarrow\CAlg(\Cat{Pr}_{\st})\) has a right adjoint,
  which is pointwise given by
  \(\cat{C}\mapsto\Ind_{\cat{I}}(\cat{C})\),
  where \(\cat{I}\) is the idealoid of very trace-class morphisms.
\end{corollary}

\begin{proof}
We write \(\cat{C}_{\rig}\) for \(\Ind_{\cat{I}}(\cat{C})\).

  First, we prove that \(\cat{C}_{\rig}\) is in \(\CAlg(\Cat{Pr}_{\st})\)
  for \(\cat{C}\in\CAlg(\Cat{Pr})\).
  We can take~\(\kappa\)
  such that \(\cat{C}\in\CAlg(\Cat{Pr}^{\kappa}_{\st})\).
  Then it follows
  from \cref{mini,x4psyd} that very nuclear \(\Ind\)-objects are in
  \(\Ind(\cat{C}_{\kappa})\).
  By \cref{xm2a0p,sequential},
  it suffices to show that
  it is also closed under cofibers.
  We consider a morphism
  between very nuclear \(\Ind\)-objects.
  By \cref{fun_ind},
  such morphism can be lifted to an \(\Ind\)-morphism
  \((C'_{i}\to C_{i})_{i}\).
  We want to prove that its cofiber \((C''_{i})_{i}\)
  is very nuclear.
  Then again,
  by \cref{sequential},
  we can assume that
  the index \(\infty\)-category is~\(\NN\).
  (The discussion above
  is valid when we replace “very nuclear” with “nuclear”.
  In that case,
  we can conclude here by noting that
  \(C''_{2n}\to C''_{2n+2}\) is trace class
  for any~\(n\),
  which follows from \cref{xjd6rf}.)
  Our current situation is that we have a diagram
  \begin{equation*}
    \begin{tikzcd}
      C'_{0}\ar[r]\ar[d]&
      C'_{1}\ar[r]\ar[d]&
      C'_{2}\ar[r]\ar[d]&
      \cdots\\
      C_{0}\ar[r]&
      C_{1}\ar[r]&
      C_{2}\ar[r]&
      \cdots
    \end{tikzcd}
  \end{equation*}
  such that \(C_{n}\to C_{n+1}\)
  and \(C'_{n}\to C'_{n+1}\)
  are very trace class.
  Then by \cref{fun_buf},
  we can use \cref{xjd6rf}
  for the diagram
  \begin{equation*}
    \begin{tikzcd}
      C'_{0}\ar[r]\ar[d]&
      C'_{1}\ar[r]\ar[d]&
      C'_{2}\ar[r]\ar[d]&
      \cdots\\
      C_{1}\ar[r]&
      C_{2}\ar[r]&
      C_{3}\ar[r]&
      \cdots
    \end{tikzcd}
  \end{equation*}
  and conclude the desired result.

  We then prove that \(\cat{C}_{\rig}\) is rigid
  for \(\cat{C}\in\CAlg(\Cat{Pr})\).
  We use \cref{rig_0}.
  First, the unit is compact in~\(\cat{C}_{\rig}\)
  by construction and \cref{xm2a0p}.
  We need to prove that every object
  is the colimit of some nuclear \(\Ind\)-object,
  i.e.,
  \(\Ind_{\cat{J}}(\cat{C}_{\rig})\to\cat{C}_{\rig}\)
  is essentially surjective,
  where \(\cat{J}\) is the ideal of trace-class morphisms in~\(\cat{C}_{\rig}\).
  By \cref{xm2a0p},
  we need to show that
  the \(\Ind\)-object \((C(n))_{n}\)
  in~\(\cat{C}\)
  such that \(C(n)\to C(n+1)\)
  is very trace class is in the image.
  By definition,
  we obtain an extension \(C\colon\QQ\cap[0,\infty)\to\cat{C}\)
  such that \(C(p)\to C(q)\) is trace class for \(p< q\).
  We left Kan extend this diagram
  to obtain a diagram \([0,\infty)\to\Ind(\cat{C})\),
  which we denote by~\(C\) by abuse of notation.
  We fix an irrational \(x\in(0,1)\)
  and consider \((C(x+n))_{n}\).
  This is an \(\Ind\)-object in~\(\cat{C}_{\rig}\).
  Moreover,
  \(C(x+n)\to C(x+n+1)\) is trace class in~\(\cat{C}_{\rig}\),
  since it factors through
  a trace-class morphism in~\(\cat{C}\).
  This shows the desired claim.

  Thus,
  we have a functor \((\X)_{\rig}\)
  that admits a natural transformation to~\(\id\).
  By \cref{rig_0} (and \cref{unique})
  when \(\cat{C}\) is rigid,
  \(\cat{C}_{\rig}\to\cat{C}\) is an equivalence.
  Therefore,
  we are left to show that
  for any \(\cat{C}\in\CAlg(\Cat{Pr}_{\st})\),
  the morphism
  \(\cat{C}_{\rig}\to\cat{C}\) becomes
  an equivalence after applying~\((\X)_{\rig}\).
  We write \(\cat{I}\)
  for the idealoid of very trace-class morphisms in~\(\cat{C}\)
  and \(\cat{J}\) for the idealoid
  of (very) trace-class morphisms in~\(\cat{C}_{\rig}\).
  We need to prove that the morphism
  \(\Ind_{\cat{J}}(\cat{C}_{\rig})\to\Ind_{\cat{I}}(\cat{C})\)
  is an equivalence,
  but this follows from the argument above.
\end{proof}

\begin{remark}\label{x52p4w}
  For \(\cat{C}\in\CAlg(\Cat{Pr}_{\st})\),
  the functor
  \(\cat{C}_{\rig}\to\cat{C}\)
  induces an equivalence on
  the full subcategories of dualizable objects.
  This can be seen by considering the right adjoints
  of the inclusions
  \begin{equation*}
    \CAlg(\Cat{Pr}^{\cg}_{\st})_{\rig}
    \hookrightarrow
    \CAlg(\Cat{Pr}_{\st})_{\rig}
    \hookrightarrow
    \CAlg(\Cat{Pr}_{\st}).
  \end{equation*}
  One can also check this directly by the description
  of \(\cat{C}_{\rig}\) in \cref{rigidify}.
\end{remark}

\subsection{Rigid spectrum}\label{ss:nr}

Finally,
we define the notion of very nuclear idempotents
and rigid spectrum.

\begin{definition}\label{xxlhf9}
  Let \(\cat{C}\)
  be a rigid stable presentably symmetric monoidal \(\infty\)-category.
  We say that
  a map of idempotents
  \(E\to E'\) is \emph{trace class}
  if there is an idempotent~\(F\)
  such that \(E'\otimes F\) is zero
  and the square
  \begin{equation}
    \label{e:rha0a}
    \begin{tikzcd}
      \unit\ar[r]\ar[d]&
      E\ar[d]\\
      F\ar[r]&
      E\otimes F
    \end{tikzcd}
  \end{equation}
  is cartesian.
\end{definition}

\begin{lemma}\label{xdi90c}
  In the situation of \cref{xxlhf9},
  any trace-class map of idempotents
  \(E\to E'\)
  underlies a trace-class morphism.
\end{lemma}

\begin{proof}
We write \([\unit,\unit]\)
  for the internal mapping object.
  We apply \([E,\X]\otimes E'\)
  to \cref{e:rha0a}
  to see that
  \([E,\unit]\otimes E'\to[E,E]\otimes E'\simeq E'\)
  is an equivalence.
  Hence we obtain a map \(\unit\to[E,\unit]\otimes E'\),
  which witnesses \(E\to E'\).
\end{proof}

\begin{definition}\label{xm9707}
  Let \(\cat{C}\)
  be a rigid stable presentably symmetric monoidal \(\infty\)-category.
  We write~\(\cat{I}\) for the idealoid
  of trace-class maps of idempotents.
  We consider the colimit functor
  \(\Ind_{\cat{I}_{\sd}}(\Idem(\cat{C}))\to\Idem(\cat{C})\).
  By \cref{xdi90c},
  we can apply \cref{unique}
  to see that this is fully faithful.
  We write \(\Idem_{\vnuc}(\cat{C})\)
  for the essential image
  and call its objects \emph{very nuclear idempotents}.
\end{definition}

\begin{remark}\label{xmzmlj}
  In \cref{xm9707},
  note that a very nuclear idempotent
  is not an idempotent
  that is inside the image of
  \(\Ind_{\cat{I}_{\sd}}(\Idem(\cat{C}))\to\cat{C}\),
  where \(\cat{I}\) is the idealoid
  of trace-class morphisms in~\(\cat{C}\).
  Since \(\cat{C}\) is rigid,
  this just recovers the notion of very Schwartz idempotent
  by \cref{rig_tc}.

  However,
  outside of the rigid situation,
  considering this notion inside algebras gives
  natural settings for certain questions;
  see \cref{xz52u7}.
\end{remark}

As in \cref{ss:sm_con},
we then prove that
\(\Idem_{\vnuc}(\cat{C})\) is a subframe
of \(\Idem(\cat{C})\)
and it is compact regular.
We use a shortcut here
by using Banaschewski–Mulvey’s construction
of the Stone–Čech compactification functor.
The first observation
is that the condition
of morphisms between idempotents
being trace-class only depends on~\(\Idem(\cat{C})\):

\begin{definition}\label{xd7840}
  In a distributive lattice,
  we say that \(p\) is \emph{rather below}~\(q\),
  denoted \(p\prec q\),
  if there is \(r\) such that
  \(p\wedge r=\bot\)
  and \(q\vee r=\top\).
\end{definition}

\begin{example}\label{xdzqx4}
  In the situation of \cref{xxlhf9},
  we see that
  \(E\to E'\) is trace class
  if and only if the corresponding
  opens in \(\Sm(\cat{C})\) satisfy the rather-below relation.
\end{example}

\begin{example}\label{xl5ls4}
  Suppose that a frame~\(R\) is compact;
  i.e., \(\top\in R\) is a compact object.
  We claim that \(x\prec y\) implies \(x\ll y\).
  Fix an element~\(z\)
  satisfying \(x\wedge z=\bot\) and \(y\vee z=\top\).
  Consider a directed subset~\(D\)
  satisfying \(y=\bigvee D\).
  Then since \(\bigvee_{d\in D}(d\vee z)=y\vee z=\top\)
  and \(\top\) is compact,
  there is \(d\in D\) such that \(d\vee z=\top\).
  Then by
  \begin{equation*}
    x
    =x\wedge\top
    =x\wedge(d\vee z)
    =(x\wedge d)\vee(x\wedge z)
    =(x\wedge d)\vee\bot
    =x\wedge d,
  \end{equation*}
  we see that \(x\leq d\).
\end{example}

We then recall
how to formalize
the compact Hausdorff property
in locale theory:

\begin{definition}\label{xy6u8d}
  We call a frame \emph{regular}
  if we have \(\bigvee_{y\prec x}y=x\)
  for any~\(x\).
\end{definition}

\begin{example}\label{xs8chv}
  If a topological space is regular,
  its frame of open subsets is regular.
\end{example}

\begin{example}\label{xjrb0o}
  By \cref{xl5ls4},
  a compact regular frame
  is stably compact and hence spatial.
  Therefore,
  the opposite of the category of compact regular frames
  is equivalent
  to the category of compact Hausdorff\footnote{In general topology,
    we do not say “compact regular”,
    since it is just equivalent to compact Hausdorff.
  } spaces~\(\Cat{CH}\).
\end{example}

The Stone–Čech compactification
also exists in locale theory,
as proven in~\cite{BanaschewskiMulvey80}\footnote{Note that they proved the existence
  of the compact regular reflection
  and compact \emph{completely} regular reflection.
  Since we assume the axiom of choice,
  we do not have to distinguish between them.
}.
We here recall the construction in~\cite[Proposition~2]{BanaschewskiMulvey80}:

\begin{theorem}[Banaschewski–Mulvey]\label{xv2gz5}
  Let \(R\) be a frame.
  Then the construction \(R\mapsto\Ind_{I_{\sd}}(R)\),
  where \(I\) denotes the idealoid of the way-below relation,
  gives the right adjoint
  to the inclusion
  of the category of compact regular frames
  to the category of frames.
\end{theorem}

\begin{remark}\label{xuw100}
  The statement of \cref{xv2gz5}
  closely parallels that of \cref{rigidify}.
  It can be proven using a similar argument.
\end{remark}

Therefore, the rigid version of \cref{sm_con_adj}
is a simple consequence of this fact:

\begin{corollary}\label{xm44hc}
  Let \(\cat{C}\)
  be a rigid stable presentably symmetric monoidal \(\infty\)-category
  and \(R\) a compact regular frame.
  Then the map
  \begin{equation*}
    \Map_{\Cat{Frm}}(R,\Idem_{\vnuc}(\cat{C}))
    \hookrightarrow
    \Map_{\Cat{Frm}}(R,\Idem(\cat{C}))
    \simeq
    \Map_{\CAlg(\Cat{Pr}_{\st})}(\Shv(R;\Cat{Sp}),\cat{C}),
  \end{equation*}
  where the second equivalence is from~\cite[Theorem~A]{ttg-sm},
  is an equivalence.
\end{corollary}

\begin{definition}\label{x92jxq}
  For a rigid stable presentably symmetric monoidal \(\infty\)-category~\(\cat{C}\),
  we write \(\Sm^{\rig}(\cat{C})\)
  for the locale whose frame of open subsets is \(\Idem_{\vnuc}(\cat{C})\)
  and call it the \emph{rigid spectrum} of~\(\cat{C}\).
\end{definition}

Rewriting \cref{xm44hc},
we obtain \cref{main_rig}.

\subsection{Computing rigid spectra}\label{ss:rig_sample}

We now compute rigid spectra in several cases.

\begin{example}\label{xuujx5}
For a quasicompact quasiseparated scheme~\(X\),
  we claim that dualizable idempotents in~\(\D(X)\)
  correspond to clopen subsets.
  Indeed, by~\cite[Proposition~2.6.0.3]{LurieSAG},
  such an idempotent is the pushforward
  of the structure sheaf of some quasicompact open subscheme of~\(X\).
  Then by~\cite[Lemma~11.1.4.6]{LurieSAG},
  we see that it must be closed too.
  This argument
  is valid more generally
  for any quasicompact quasiseparated algebraic space,
  but in fact, it can be deduced formally,
  see \cref{xqp9qd}.

  By this, we obtain a morphism of locales
  \(\Sm^{\rig}(\D(X))\to\pi_{0}\lvert X\rvert\),
  where the Stone space \(\pi_{0}\lvert X\rvert\)
  is called the \emph{Pierce spectrum}.
  This morphism need not be an isomorphism in general;
  see \cref{xu0zbv,xo25m9,xav7u6}.
\end{example}

The following observation arose in a discussion with Rodríguez~Camargo:

\begin{example}\label{xqp9qd}
  Let \(E\) be an idempotent algebra
  in a stably symmetric monoidal \(\infty\)-category~\(\cat{C}\)
  whose underlying object is dualizable.
  We claim that \(E\) is complementable.
  Indeed, the dual \(E^{\vee}\) naturally carries an \(E\)-module structure,
  so there is an equivalence \(E^{\vee} \simeq E^{\vee} \otimes E\).
  Dualizing this yields \(E \simeq E^{\vee}\),
  from which we obtain a canonical morphism \(E \to \unit\)
  exhibiting \(E\) as an idempotent coalgebra.
  The idempotent algebra \(\cofib(E \to \unit)\)
  complements~\(E\).
\end{example}

Computing \(\Sm^{\rig}\) is
obviously easier than \(\Sm^{\con}\),
since very nuclear idempotents
are very Schwartz.
For example,
\cref{xgbjvz} implies the following:

\begin{corollary}\label{x56p8x}
  Let \(X\) be a noetherian scheme.
  Then \(\Sm^{\rig}(\D(X))\) is the (finite) set
  of connected components.
\end{corollary}

Moreover,
\cref{xw2z2y} implies the following,
which proves \cref{tan_ch}:

\begin{corollary}\label{xa1a4f}
  For a compact Hausdorff space~\(X\),
  the morphism \(\Sm^{\rig}(\Shv(X;\Cat{Sp}))\to X\)
  is an equivalence.
\end{corollary}

We study some more examples:

\begin{example}\label{xu0zbv}
We here show that
  there is a domain~\(A\) that is not a field
  such that \(\Frac(A)\) is a very nuclear idempotent.
  This is analogous to \cref{xrxo97},
  but this time,
  we take \(\Gamma=\bigoplus_{p\in P}\ZZ[1/2]\)
  and consider \(V=k\llbracket T^{\Gamma_{\geq0}}\rrbracket\).
  We claim that
  \(V[T^{-p}]\to V[T^{-p'}]\)
  for \(p<p'\) is a trace-class map
  of idempotents in \(\D(V)\).
  Indeed, \(V/\langle T^{p/2^{n}}\mid n\geq0\rangle\)
  is an idempotent
  witnessing its trace-class property.
\end{example}

\begin{example}\label{xo25m9}
For a compact Hausdorff space~\(X\),
  we write \(\Cls{C}(X)\)
  for the ring of continuous real-valued functions on~\(X\).
  It is well known
  that the kernel~\(I\)
  of \(\Cls{C}(X)\to\Cls{C}(Z)\)
  is a flat \(\Cls{C}(X)\)-module with \(I^{2}=I\);
  see, e.g.,~\cite[Theorem~6.13]{k-ros-1}.
  From this,
  we observe that
  there is a surjection
  \(\Sm^{\rig}(\D(\Cls{C}(X)))\to X\).
\end{example}

\begin{example}\label{xav7u6}
This is similar to \cref{xo25m9},
  but here we argue more concretely
  using principal localizations.
  Consider \(A=\Cls{C}([0,1])\).
  For \(p\in[0,1)\), we consider
  \begin{equation*}
    f_{p}(x)=
    \begin{cases}
      0&\text{if \(x\in[0,p]\),}\\
      \frac{x-p}{1-p}&\text{if \(x\in[p,1]\).}\\
    \end{cases}
  \end{equation*}
  We claim that \(A[f_{p}^{-1}]\to A[f_{p'}^{-1}]\)
  is a trace-class map of idempotents
  for \(p<p'\).
  To see this,
  we set \(q=(p+p')/2\)
  and consider \(g(x)=f_{1-q}(1-x)\).
  Then \(A[g^{-1}]\) witnesses
  the trace-class property.
\end{example}

\begin{remark}\label{xmd5ll}
  Note that in \cref{xav7u6},
  we only used the Zariski idempotents here.
  A similar consideration shows
  that \(\beta\) of a spectral space
  does not coincide with~\(\pi_{0}\)
  (and hence not totally disconnected).
\end{remark}

\subsection{Digression: very nuclear algebras and rigidification}\label{ss:rig_thm}

We here demonstrate that the general notion
of very nuclear algebras is important:

\begin{theorem}\label{xz52u7}
  Let \(\cat{C}\) be an object of \(\CAlg(\Cat{Pr}^{\cg}_{\st})\).
  Suppose that
  \(A\) is the colimit of
  an object in \(\Ind_{\cat{I}_{\sd}}(\CAlg(\cat{C}))\),
  where \(\cat{I}\) is the pullback
  of the idealoid of trace-class morphisms
  in~\(\cat{C}\)
  along the forgetful functor \(\CAlg(\cat{C})\to\cat{C}\).
  Then the canonical functor
  \begin{equation*}
    \Mod_{A}(\cat{C}_{\rig})\to(\Mod_{A}(\cat{C}))_{\rig}
  \end{equation*}
  is an equivalence.
\end{theorem}

\begin{remark}\label{xxoacl}
  In \cref{xz52u7},
  the assumption on~\(A\) is stronger
  than what is used in the proof:
  For example,
  transitions can be~\(\E_{1}\) instead of~\(\E_{\infty}\).
  We only treat this case
  for simplicity
  and leave such generalizations
  to the reader.
\end{remark}

We see some geometric examples
to motivate this theorem:

\begin{example}\label{ana_rig}
  Let \(K\) be a Stein compact set.
  Then the commutative algebra object \(\shf{O}(K)\)
  in \(\D(\CC_{\liq})\)
  (see \cref{x2xgqq})
  satisfies the condition of \cref{xz52u7}; see~\cite{Complex}.
However, as noted in \cref{ana_cg},
  we do \emph{not} expect \(\D(\CC_{\liq})\)
  to lie in \(\CAlg(\Cat{Pr}^{\cg}_{\st})\).
  Nevertheless,
  the “sequential” part of this category does,
  and it suffices to apply \cref{xz52u7}
  to conclude that the canonical functor
  \begin{equation*}
    \Mod_{\shf{O}(K)}(\D(\CC_{\liq})_{\rig})
    \to
    \D(\shf{O}(K)_{\liq})_{\rig}
  \end{equation*}
  is an equivalence.
  A detailed account of this will be given elsewhere.

  The same remarks apply if we replace
  the liquid structure with the gaseous structure
  (and heavy condensed mathematics
  with the light condensed mathematics) of~\cite{AnaSta}.
\end{example}

\begin{example}\label{xourh2}
  We here explain a nonexample relevant to \cref{ana_rig}.
  Instead of \(\cat{O}(K)\),
  we take~\(\CC[T]\).
  We claim that
  the canonical functor
  \(\Mod_{\CC[T]}(\D(\CC_{\liq})_{\rig})
  \to\D(\CC[T]_{\liq})_{\rig}\)
  is not an equivalence.
  Indeed, the dualizable object of \(\D(\CC[T])_{\liq}\)
  constructed in~\cite[Example~9.6]{Complex},
  which determines a dualizable object
  by \cref{x52p4w},
  is not in the image.
\end{example}

We then move on to proving \cref{xz52u7}.
Discussions with Scholze were helpful in developing this proof.

\begin{lemma}\label{xauqm1}
  Consider \(\cat{C}\in\CAlg(\Cat{Pr}^{\cg}_{\st})\)
  and an object \(A\in\CAlg(\cat{C})\).
  Let \(B^{(\X)}\colon\QQ\cap[0,\infty)\to\cat{C}\)
  be a diagram such that \(B^{p}\to B^{q}\) is compact
  for \(p<q\)
  and \(B^{0}\) is the underlying object of~\(A\).
  When \(M\) is a compact \(A\)-module,
  \(B^{p}\otimes_{A}M\to B^{q}\otimes_{A}M\) is compact in~\(\cat{C}\)
  for \(p<q\).
\end{lemma}

\begin{proof}
  We consider the full subcategory \(\cat{D}\subset\Mod_{A}(\cat{C})\)
  spanned by~\(M\)
  such that 
  \(B^{p}\otimes_{A}M\to B^{q}\otimes_{A}M\) is compact in~\(\cat{C}\)
  for \(p<q\).
  We want to show that \(\cat{D}\) contains
  all compact objects.
  Clearly,
  \(\cat{D}\) is closed under retracts and 
  shifts.
  It also contains \(A\otimes C\)
  for any compact object \(C\in\cat{C}\).
  Therefore,
  we are left to show that
  \(\cat{D}\) is closed under cofibers.
  We consider the cofiber sequence \(M'\to M\to M''\)
  with \(M\) and \(M'\in\cat{D}\).
  Since 
  \(B^{p}\otimes_{A}M'\to B^{(p+q)/2}\otimes_{A}M'\)
  and 
  \(B^{(p+q)/2}\otimes_{A}M'\to B^{q}\otimes_{A}M'\)
  are compact by assumption,
  we see \(M''\in\cat{D}\) from \cref{xbagrg}.
\end{proof}

\begin{lemma}\label{xmxx8h}
  Consider \(\cat{C}\in\CAlg(\Cat{Pr}^{\cg}_{\st})\)
  and an object \(A\in\CAlg(\cat{C})\).
  Suppose that \(A\) underlies a nuclear object in~\(\cat{C}\).
  Then an \(A\)-module~\(M\) is nuclear
  in \(\Mod_{A}(\cat{C})\)
  if and only if it is nuclear in~\(\cat{C}\).
\end{lemma}

\begin{proof}
  In this proof,
  we use~\cref{i:2hh2s}
  of \cref{xzlma5} as the definition of nuclearity.

  Since \(A\) is nuclear,
  for any compact object \(C_{0}\in\cat{C}\),
  the morphism
  \([C_{0},\unit]\otimes A\to[C_{0},A]\)
  is an equivalence,
  where \([\X,\X]\)
  denotes the internal mapping object.

  We prove the “only if” direction.
  It suffices to prove that
  \([C_{0},\unit]\otimes M\to[C_{0},M]\)
  is an equivalence for any compact object \(C_{0}\in\cat{C}\).
  From the equivalences
  \begin{equation*}
    [C_{0},\unit]\otimes M
    \simeq
    ([C_{0},\unit]\otimes A)\otimes_{A}M
    \simeq
    [C_{0},A]\otimes_{A}M,
  \end{equation*}
  we obtain the desired 
  result by applying the nuclearity of~\(M\)
  as an \(A\)-module for the compact \(A\)-module \(A\otimes C_{0}\).

  We prove the “if” direction.
  We consider the full subcategory \(\cat{D}\subset\Mod_{A}(\cat{C})\)
  spanned by~\(M_{0}\) such that
  \([M_{0},A]_{A}\otimes_{A}M\to[M_{0},M]_{A}\)
  is an equivalence,
  where \([\X,\X]_{A}\)
  denotes the internal mapping object in \(\Mod_{A}(\cat{C})\).
  We want to show that \(\cat{D}\) contains
  all compact \(A\)-modules.
  Since it is immediate that \(\cat{D}\) is closed under
  retracts and finite colimits,
  it suffices to show that
  \(A\otimes C_{0}\in\cat{D}\)
  when \(C_{0}\) is a compact object in~\(\cat{C}\).
  The argument above shows the desired claim.
\end{proof}

\begin{lemma}\label{x4gx1u}
  Let \(\cat{C}\) be an object of \(\CAlg(\Cat{Pr}^{\cg}_{\st})\).
  A diagram
  \(A^{(\X)}\colon\QQ\cap[0,\infty)\to\CAlg(\cat{C})\)
  is such that
  \(A^{p}\to A^{q}\) underlies
  a trace-class morphism in~\(\cat{C}\) for \(p<q\).
  We write \(A=A^{0}\).
  A diagram
  \(M\colon\QQ\cap[0,\infty)\to\Mod_{A}(\cat{C})\)
  is such that
  \(M(p)\to M(q)\) is trace class in \(\Mod_{A}(\cat{C})\)
  for \(p<q\).
  Then
  \(M'=\injlim_{p}A^{p}\otimes_{A^{0}}M(p)\)
  is very nuclear over~\(\unit\).
\end{lemma}

\begin{proof}
  By Kan extension and irrational shifting
  as in the proof of \cref{rigidify},
  we can assume that
  the underlying object of~\(A(p)\) is very nuclear.
  Similarly,
  we can assume that
  \(M(p)\) is (very) nuclear 
  in \(\Mod_{A}(\cat{C})\).

  For any~\(p\),
  we see that \(A^{p}\otimes_{A^{0}}M(p)\)
  is nuclear in \(\Mod_{A^{p}}(\cat{C})\)
  and hence in~\(\cat{C}\) by \cref{xmxx8h}.
  Therefore,
  it suffices to show that
  \(A^{p}\otimes_{A^{0}}M(p)\to A^{q}\otimes_{A^{0}}M(q)\)
  is compact in~\(\cat{C}\)
  for \(p<q\).
  For that,
  note that
  \(A^{p}\otimes_{A^{0}}M(p)\to A^{p}\otimes_{A^{0}}M(q)\)
  is trace class in \(\Mod_{A^{p}}(\cat{C})\)
  and hence factors through a compact \(A^{p}\)-module.
  Now the desired result follows from \cref{xauqm1}.
\end{proof}

\begin{proof}[Proof of \cref{xz52u7}]
  By \cref{sequential},
  it suffices to consider
  the case
  when \(A=\injlim_{n}A^{n}\)
  such that \(A^{n}\to A^{n+1}\)
  is in~\(\cat{I}_{\sd}\).
  We extend this diagram
  to \(A^{(\X)}\colon\QQ\cap[0,\infty)\to\CAlg(\cat{C})\)
  such that \(A^{p}\to A^{q}\) in~\(\cat{I}_{\sd}\).
  For \(p>0\),
  we write \(A^{<p}\) for \(\injlim_{q<p}A^{q}\),
  which is an object of~\(\cat{C}_{\rig}\)
  by assumption.
  Since both sides are full subcategories
  of \(\Mod_{A}(\cat{C})\),
  the functor \(\Mod_{A}(\cat{C}_{\rig})
  \to\Mod_{A}(\cat{C})_{\rig}
  \) is fully faithful.
  For \(p>0\),
  we prove the existence of the factorization
  \begin{equation*}
    \begin{tikzcd}
      \Mod_{A^{<p}}(\cat{C}_{\rig})\ar[r]\ar[d,hook]&
      \Mod_{A^{<p+1}}(\cat{C}_{\rig})\ar[d,hook]\\
      (\Mod_{A^{<p}}(\cat{C}))_{\rig}\ar[r]\ar[ur,dashed]&
      (\Mod_{A^{<p+1}}(\cat{C}))_{\rig}\rlap;
    \end{tikzcd}
  \end{equation*}
  by \cref{xsoyii},
  this implies the desired claim.
  We prove this by considering a generator;
  by \cref{sequential},
  it suffices to show that
  \(M\colon\QQ\cap[0,\infty)\to\Mod_{A^{<p}}(\cat{C})\)
  such that \(M(p')\to M(q')\)
  is trace class over \(A^{<p}\) for \(p'<q'\),
  its base change underlies a very nuclear object in~\(\cat{C}\).
  This statement follows from \cref{x4gx1u}.
\end{proof}

\bibliographystyle{plain}
\let\SS\oldSS  \newcommand{\yyyy}[1]{}

\end{document}